\documentclass[11 pt]{amsart}

\usepackage[latin1]{inputenc}
\usepackage{amsmath,amsthm,amssymb,amscd,mathrsfs,mathtools,color}
\definecolor{darkgreen}{rgb}{0.0, 0.6, 0.13}

\newcommand{\legendre}[2]{%
\genfrac(){}{}{#1}{#2}
}

\newtheorem{thm}{Theorem}[section]
\newtheorem{conj}{Conjecture}[section]
 \newtheorem{cor}[thm]{Corollary}
 \newtheorem{lem}[thm]{Lemma}
 \newtheorem{prop}[thm]{Proposition}

 \theoremstyle{definition}
 
 \theoremstyle{remark}
 \newtheorem{rem}[thm]{Remark}
 
 \numberwithin{equation}{section}

\def\sqw{\hbox{\rlap{\leavevmode\raise.3ex\hbox{$\sqcap$}}$%
\sqcup$}}
\def\findem{\ifmmode\sqw\else{\ifhmode\unskip\fi\nobreak\hfil
\penalty50\hskip1em\null\nobreak\hfil\sqw
\parfillskip=0pt\finalhyphendemerits=0\endgraf}\fi}

\newcommand{\R}{\mathbb R}

\newcommand{\Q}{\mathbb Q}
\newcommand{\N}{\mathbb N}
\newcommand{\Z}{\mathbb Z}

\newcommand{\T}{{\bf T}}

\begin{document}

\title[Strichartz estimates for the Schr\"odinger equation on non-rectangular tori]{Strichartz estimates for the Schr\"odinger equation on non-rectangular two-dimensional tori}

\author{Yu Deng}
\address[Y.\@ Deng]{Courant Institute of Mathematical Sciences\\ New York University\\ 251 Mercer Street\\ New York, N.Y. 10012-1185\\ USA}
\email{yudeng@cims.nyu.edu}

\author{Pierre Germain}
\address[P.\@ Germain]{Courant Institute of Mathematical Sciences\\ New York University\\ 251 Mercer Street\\ New York, N.Y. 10012-1185\\ USA}
\email{pgermain@cims.nyu.edu}

\author{Larry Guth}
\address[L. Guth]{Massachusetts Institute of Technology \\ Department of Mathematics \\
 77 Massachusetts Avenue, Cambridge, MA 02139-4307\\ USA}
\email{larry.guth.work@gmail.com}

\author{Simon L.\@ Rydin Myerson}
\address[S.\@ L.\@ Rydin Myerson]{Mathematisches Institut\\Georg-August-Universit\"at G\"ottingen\\Bunsenstra\ss e 3-5\\D-37073 G\"ottingen\\Deutschland
}
\email{myerson@goettingen.de}

\maketitle

\begin{abstract}We propose a conjecture for long time Strichartz estimates on generic (non-rectangular) flat tori. We proceed to partially prove it in dimension 2. Our arguments involve on the one hand Weyl bounds; and on the other hands bounds on the number of solutions of Diophantine problems.\end{abstract}

\section{Introduction}

\subsection{A general question} 

It is a classical result that, if $v$ is a solution of the linear Schr\"odinger equation on $\mathbb{R}^d$ with data $v_0$, 
$$
i \partial_t v -\Delta_M v = 0,\qquad
v(0,x)=v_0(x),
$$
and if furthermore $P_N$ is a projector on frequencies $\lesssim N$, then
\begin{equation}
\label{StrichartzRd}
\| P_N v \|_{L^p(\mathbb{R} \times \mathbb{R}^d)} \lesssim N^s \| v_0\|_{L^2(\mathbb{R}^d)} \qquad \mbox{with $s= \max \left( 0,\frac{d}{2} - \frac{d+2}{p} \right)$}
\end{equation}
{(Strichartz estimates for the linear Schr\"odinger equation are usually stated with different space and time Lebesgue indices, see~\cite{Tao}; the above inequality follows from these through the Sobolev embedding theorem).}

\medskip

Can this inequality be extended to compact manifolds? Let $M$ be a compact Riemannian manifold without boundary, let $v_0\in L^2(M)$ and let $v$ be the solution  to the linear Schr\"odinger equation 
\begin{equation}\label{mock-duck}
i \partial_t v -\Delta_M v = 0,\qquad
v(0,x)=v_0(x)
\end{equation}
where $\Delta_M$ is the Laplace-Beltrami operator on $M$. 
Finally, let $P_N$ be a projector on eigenmodes of the Laplacian $< N^2$, for instance $P_N = \chi \left( \frac{-\Delta_M}{N^2} \right)$, where $\chi$ is a smooth cutoff function.

In order to extend the inequality~\eqref{StrichartzRd}, a natural question is to determine the best constant $C(T,p,N,M)$ in
\begin{equation}
\label{thequestion}
\| P_N v \|_{L^p([0,T] \times M)} \leq C(T,p,N,M) \| v_0 \|_{L^2(M)}.
\end{equation}
In cases where such a problem could be solved, experience shows that the dependence of $C$ on $T,p,N$ is in general of the type $N^{s_1(p)} T^{s_2(p)}$, or a sum of such summands, up to possible subpolynomial factors $O_\epsilon(N^{\epsilon}T^\epsilon)$.

Writing $\lVert v_0 \rVert_{H^s}$ for the Sobolev norm (defined through fractional powers of $\Delta_M$), it is essentially equivalent to examine inequalities of the type $\| P_N v \|_{L^p([0,T] \times M)} \leq C \| v_0 \|_{H^s(M)}$.

\subsection{Background for $T=1$}

Burq, G\'erard and Tzvetkov in~\cite{BGT}, showed that, for general $d$-dimensional compact Riemannian manifolds,
$$
\| P_N v \|_{L^p([0,1] \times M)} \lesssim N^{s(p)} \| v_0 \|_{L^2(M)} \qquad \mbox{with $s(p) = \max\left( \frac{d}{4} - \frac{d}{2p}, \frac{d}{2} - \frac{d+1}{p} \right)$}.
$$
Notice that these authors actually proved more general inequalities {(Theorem 1 in ~\cite{BGT})}, allowing for a different Lebesgue index in time and space. As a consequence of the above estimate, they 
deduce existence results for certain nonlinear Schr\"odinger equations in two and three dimensions.

Is the above estimate optimal for some $M$? This seems unclear. In the case of the sphere (or more generally of Zoll manifolds), which might lead to the largest constants $C(T,p,N,M)$ (in particular since it has very concentrated eigenmodes of the Laplacian), the above authors show {(Theorem 4 in~\cite{BGT})} that the above estimate can be improved, if $p=4$, to
$$
s(4) = \max\left( \frac{1}{8},\frac{d}{4}-\frac{1}{2}\right);
$$
this is furthermore optimal if $d \geq 3$.

The best-studied situation is when $M$ is a flat torus $\R^d/\Lambda$ for some rank $d$ lattice $\Lambda$. In the foundational paper of Bourgain~\cite{B1}, Strichartz estimates lead to well-posedness results for the nonlinear Schr\"odinger equation on the square torus $\mathbb{T}^d=\R^d/\Z^d$. This led to much subsequent work on the nonlinear problem, see for example~\cite{HTT} and~\cite{IP} for the energy-critical case. Coming back to (linear) Strichartz estimates, Bourgain stated a conjecture~\cite[(3.2)-(3.4), cf.\@ (1.7)]{B1} for these Strichartz estimates which was  pursued in a number of works~\cite{B2,B3,GOW}, culminating in  Theorem~2.2 of Bourgain and Demeter~\cite{BD1} which delivers an essentially sharp range 
\begin{equation}\label{resplendent-quetzal}
\| P_N v \|_{L^p([0,1] \times M)} \lesssim N^{s(p)} \| v_0 \|_{L^2(M)} \qquad \mbox{with $s(p) > \max\left( 0, \frac{d}{2} - \frac{d+2}{p} \right)$}.
\end{equation}
for general tori. 

When $M$ is a flat torus and $T$ is large, we study refinements of \eqref{resplendent-quetzal} and their relation to the geometry of $M$. Note that this case is significant in connection with \emph{weak turbulence} for nonlinear Schr\"odinger equations~\cite{DG,D}. 
We conjecture that the optimal bound is sensitive to the geometry of the torus $M$, depending on whether $M$ is rectangular ($M=\R/\theta_i\Z$) and also on whether the spectrum of $\Delta_M$ the torus is rational, in the sense that all ratios of eigenvalues are rational numbers. We prove a lower bound supporting this conjecture, and in the two-dimensional case we give partial results toward it.

For the rest of this paper we restrict to the case when $M$ is a flat torus.

\subsection{Strichartz estimates on tori for $T \geq 1$}  Define a torus
$$
M =
M_{e_1 \dots e_d} = \mathbb{R}^d / (\mathbb{Z} e_1 + \dots + \mathbb{Z} e_d),
$$
where $e_1, \dots, e_d$ are {linearly independent} vectors of $\mathbb{R}^d$. As above, we ask for the best constant $C(p,T,N,M)$ in
\begin{equation}\label{red-ruffed-lemur}
\| e^{it\Delta_M} f \|_{L^p([0,T] \times M)} \leq C(p,T,N,M)  \| f \|_{L^2 (M)},
\end{equation}
where  $T \geq 1$ and the Fourier transform $\hat{f}$ is supported on the ball  $B(0,N)$.

The three first authors~\cite{DGG} investigated the case $T \geq 1$ in the ``rectangular" case where $(e_1,\dots,e_d)$ form an orthogonal basis; we focus in the present article on the case where $(e_1,\dots,e_d)$ are in general position.

We change variables as follows. If $x\in M$ then write
$$
x = y_1 e_1 + \dots + y_d e_d
$$
(so that now $y \in \mathbb{T}^d = \mathbb{R}^d / \mathbb{Z}^d$) and expand in Fourier series
$$
f(x) = \sum \widehat{f}_k e^{2\pi i k \cdot y}.
$$
Then $\nabla_x = A \nabla_y$ for a matrix $A$, and
$$
\Delta_M f(x) = - 4 \pi^2 \sum_{k \in \mathbb{Z}^d} |Ak|^2  \widehat{f}_k e^{2\pi i k \cdot y}.
$$
In other words, it suffices to consider the inequality
$$
\| e^{it\widetilde{\Delta}} f \|_{L^p([0,T] \times \T^d)} \leq C(p,T,N,M) \| f \|_{L^2 (\T^d)},
$$
where 
$$
\widetilde{\Delta} = \frac{1}{2\pi} \sum \alpha_{ij} \partial_i \partial_j, \qquad  \mbox{for a symmetric, positive definite, matrix $(\alpha_{ij})= 2\pi A^TA$}. 
$$
Without loss of generality we set
$
\alpha_{11} = 1.
$

By dividing the range for $t$ into pieces of length 1, the result of Bourgain and Demeter referred to around \eqref{resplendent-quetzal} implies that for $p\geq 2$ and $\epsilon>0$ we have
\begin{equation}\label{adelomyrmex-quetzal}
\| e^{it\widetilde{\Delta}} f \|_{L^p([0,T] \times \T^d)} \lesssim_{p,(\alpha_{ij}),\epsilon} T^{1/p}N^{\epsilon}(1+N^{\frac{d}{2}-\frac{d+2}{p}}) \| f \|_{L^2 (\T^d)}.
\end{equation}
In this formulation, we ask whether the factor $T^{1/p}$ can be improved for large $T$.

If the $\alpha_{ij}$ are rational (equivalently: if the eigenvalues of $\Delta_M$ all lie in $\Q$) then   the operator $ e^{it\widetilde{\Delta}}$ is periodic, and consequently $T^{1/p}$ is the correct growth rate as $T\to \infty$. If the $\alpha_{ij}$ were unusually well approximated by rational numbers, for example if they were Liouville numbers, then one would expect a similar behavior.

If the $\alpha_{ij}$ are not well approximable by rational numbers then one expects $C$ to grow more slowly than $T^{1/p}$. In~\cite{DGG} the first three authors conjecture that if $(\alpha_{ij})$ is diagonal then
$$
\| e^{it\widetilde{\Delta}} f \|_{L^p([0,T] \times \T^d)} \lesssim_{p,(\alpha_{ij}),\epsilon} N^{\epsilon}(T^{1/p}+N^{\frac{d}{2}-\frac{d+2}{p}}+T^{1/p}N^{\frac{d}{2}-\frac{3d}{p}}) \| f \|_{L^2(\T^d)},
$$
provided $(\alpha_{ij})$ is \emph{generic}; that is provided $(\alpha_{ij})$ is lies outside an exceptional set of diagonal matrices $E$ which has $d$-dimensional Lebesgue measure zero. The following conjecture proposes a stronger result for general, not necessarily diagonal matrices $(\alpha_{ij})$.

Let $E$ be the set of all symmetric $(\alpha_{ij})$ such that $-2 \leq \alpha_{ij} \leq 2$, and with smallest absolute eigenvalue $|\lambda_1|>1$. In the remainder of this paper we say that a statement holds for generic $(\alpha_{ij})$ if it is true for all $(\alpha_{ij}) \in E\setminus F$, where $F$ has Lebesgue measure zero. Implicit constants might of course not be uniform in $(\alpha_{ij})$.

\begin{conj} \label{theconjecture}
	For $p\geq 2$, generic $(\alpha_{ij})$, and  any $\epsilon > 0$,
	$$
		\| e^{it {\widetilde{\Delta}}} f \|_{L^p ([0,T]  \times \mathbb{T}^d)} \lesssim N^\epsilon \left[ N^{\frac{d}{2} - \frac{d+2}{p}} + T^{\frac{1}{p}} \sum_{n=0}^d N^{\frac{n}{2} - \frac{n^2 + 2n}{p}}  \right] \| f \|_{L^2(\mathbb{T}^d)},
	$$
provided $\operatorname{Supp} \widehat{f} \subset B(0,N)$.
\end{conj}

This can also be written
\begin{multline*}
\| e^{it {\widetilde{\Delta}}} f \|_{L^p ([0,T] \times \mathbb{T}^d)}
\lesssim
\\
N^\epsilon \| f \|_{L^2(\mathbb{T}^d)}
\left\{
\begin{array}{ll}
T^{\frac{1}{p}} & \mbox{if $2\leq p\leq \frac{2(d+2)}{d}$} \\
N^{\frac{d}{2} - \frac{d+2}{p}} + T^{\frac{1}{p}} & \mbox{if $\frac{2(d+2)}{d}<p\leq 6$} \\
N^{\frac{d}{2} - \frac{d+2}{p}} + T^{\frac{1}{p}} N^{\frac{n}{2} - \frac{n^2 + 2n}{p}} & \mbox{if $4n+2<p\leq 4n+6$, with $n \in \{1,\dots,d-1\}$} \\
N^{\frac{d}{2} - \frac{d+2}{p}}+ T^{\frac{1}{p}} N^{\frac{d}{2} - \frac{d^2 + 2d}{p}}  & \mbox{if $p > 4d+2$}
\end{array}
\right.
\end{multline*}

The heuristic behind this conjecture, and some supporting lower bounds, are explained in section~\ref{conjecture-justification} below. In particular, it is seen there that the saving over the rectangular case is due to the longer \emph{refocusing time}, which more closely reflects the expected behavior on a ``typical" nonpositively curved manifold.

We now focus on the case of dimension $d=2$, in which case the above conjecture becomes:
\begin{itemize}
\item For $p \leq 4$, bound $\sim T^{\frac{1}{p}}$.
\item For $4<p\leq 6$, bound $\sim T^{\frac{1}{p}}+N^{1-\frac{4}{p}}$ (with a cutoff at $T \sim N^{p-4}$)
\item For $6<p\leq 10$, bound $\sim T^{\frac{1}{p}} N^{\frac{1}{2}-\frac{3}{p}} + N^{1-\frac{4}{p}}$ (with a cutoff at $T \sim N^{\frac{p}{2}-1}$)
\item For $p>10$, bound $\sim N^{1-\frac{8}{p}} T^{\frac{1}{p}} + N^{1-\frac{4}{p}}$ (with a cutoff at $T \sim N^{4}$).
\end{itemize}

From \eqref{adelomyrmex-quetzal} one immediately obtains the conjecture for $p<4$. Our results for the remaining cases are as follows. 

\begin{thm}
\label{thmweyl}
If $d=2$, for generic $(\alpha_{ij})$ and for $p>4$,
$$
\| e^{it\widetilde{\Delta}} f \|_{L^p([0,T] \times \mathbb{T}^2)}  \lesssim_{\alpha, \epsilon}
N^\epsilon \| f \|_{L^2(\mathbb{T}^d)} 
\left[ N^{1-\frac{4}{p}} + T^{\frac{1}{p}} N^{\frac{2}{3}(1-\frac{4}{p})} \right],
$$
provided $\operatorname{Supp} \widehat{f} \subset B(0,N)$.\end{thm}

The theorem, which is proved in sections~\ref{section2} and~\ref{section3}, thus misses the conjecture {when $4<p<6$ and $T >N^{p-4}$} by a factor of $N^{\frac{2}{3}(1-\frac{4}{p})}$. {Proving Strichartz estimates can be reduced to counting solutions of Diophantine inequalities when the Lebesgue index is an even integer. Using this idea, the cases $p=8,10$ can be addressed, and they are the focus of the next theorem; but the case $p=6$ seems to remain out of reach.}

\begin{thm} 
\label{thmp8} The conjecture holds for $d = 2$, $p \geq 8$. In other words, there holds, for generic $(\alpha_{ij})$ and all $\epsilon>0$,
\begin{align*}
&\mbox{if $8 \leq p \leq 10$}, & \| e^{it\widetilde{\Delta}} f \|_{L^p([0,T] \times \mathbb{T}^2)}
\lesssim_{\alpha, \epsilon} N^\epsilon \| f \|_{L^2(\mathbb{T}^d)} \left[ N^{1-\frac{4}{p}} + T^{\frac{1}{p}} N^{\frac{1}{2}-\frac{3}{p}} \right]\\
&\mbox{if $p > 10$},  & \| e^{it\widetilde{\Delta}} f \|_{L^p([0,T] \times \mathbb{T}^2)}
\lesssim_{\alpha, \epsilon} N^\epsilon  \| f \|_{L^2(\mathbb{T}^d)} \left[ N^{1-\frac{4}{p}} + N^{1-\frac{8}{p}} T^{\frac{1}{p}} \right] ,
\end{align*}
provided $\operatorname{Supp} \widehat{f} \subset B(0,N)$
\end{thm}

\begin{rem} Most of the arguments developed in the present article apply to the case when the "elliptic" Laplacian $\partial_x^2 + \partial_y^2$ is replaced by a "hyperbolic" Laplacian $\partial_x^2 - \partial_y^2$; or in other words, when the symmetric matrix $(\alpha_{ij})$ is allowed to be indefinite, but remains non-degenerate. However, the $\ell^2$ decoupling inequality follows a different numerology, see~\cite{BD3}. As a consequence, the results for $p \geq 8$ are identical for elliptic and hyperbolic Laplacian, but for $p \leq 8$ a different set of exponents is found. \end{rem}

\subsubsection{The Weyl bound argument}\label{sec:weyl_arg}
This is used to prove Theorem~\ref{thmweyl}. Let $d=2$. With $
f(x) = \sum_{k \in \mathbb{Z}^2} \widehat{f}_k e^{2\pi i k \cdot x},$ we have
\begin{equation}\label{bigbird}
\| e^{it\widetilde{\Delta}} f \|^p_{L^p([0,T] \times \mathbb{T}^2)}
= \int_0^T \int_{\mathbb{T}^2} \bigg\lvert  \sum_{n \in \mathbb{Z}^2} \hat{f}_n e^{2\pi i \left[x \cdot n+ t \sum_{i,j} \alpha_{ij} n_i n_j \right]} \bigg\rvert^{p}\,dxdt.
\end{equation}
We construct an exponential sum like that inside the absolute value, with the $\hat{f}_n$ replaced by some more tractable coefficients. Let  $\chi:\R^2\to[0,\infty)$ be a compactly supported smooth cutoff function. Let $K_N$ denote the exponential sum
$$
K_N(t,x) = \sum_{n \in \mathbb{Z}^2} \chi \left( \frac{n_1}{N} \right) \chi \left( \frac{n_2}{N} \right) e^{2\pi i \left[x \cdot n+ t \sum_{i,j} \alpha_{ij} n_i n_j \right]},
$$
which can be regarded as a regularized fundamental solution of $(i \partial_t - \widetilde \Delta) u = 0$.

Similar sums are investigated by Bentkus and G\"otze~\cite{BG1,BG2}, G\"otze~\cite{Goe}, M\"uller~\cite{Mue} and G\"otze and Margulis~\cite{GM} in the case when $(\alpha_{ij})$ is irrational, that is not a multiple of a matrix with rational entries. We investigate this sum for generic $(\alpha_{ij})$. In Lemma~\ref{dispersive} and Proposition~\ref{minkowski} we will prove the following bounds: 

\begin{align*}
\sup_{x \in \mathbb{T}^2} |K_N(t,x)| &\lesssim \min  \left( N^2 ,\frac{1}{t} \right)
&&(T\lesssim \frac{1}{N}),
\\
\sup_{x \in \mathbb{T}^2} |K_N(t,x)|
&\preceq N^{4/3} (1+ t^{1/6})
&&(T> \frac{1}{N}),
\\
\int_1^T \sup_{x \in \mathbb{T}^2} |K_N(t,x)|^{4} \, dt &\preceq
N^{4}T
&&(T>1).
\end{align*}
We prove these estimates using the geometry of numbers approach pioneered by Davenport~\cite{Davenport58}. The first one is classical and optimal. An analysis similar to step (2) in section~\ref{conjecture-justification} suggests that $N(1+t^{1/4})$ might be the correct bound in the second formula. The third bound can be related to a system of diagonal equations to see that on average over $\alpha_{ij}$ it is optimal, see also section~\ref{sec:d>2} below.

The classical $TT^*$ argument of Stein-Tomas can be localized to obtain level-set estimates for $ |e^{it\widetilde \Delta} f (x,t)|$. Consider $f$ a function on the torus, normalized in $L^2$, and localized in Fourier on $B(0,N)$, and let
$$
E_\lambda = \{ (x,t) \in \mathbb{T}^2 \times [-T,T] \;\; \mbox{s.t.} \;\; |e^{it\widetilde \Delta} f (x,t)|>\lambda \}.
$$
Then, if  $\chi(t/T) K_N(t,x)$ is split into
$$
\chi(t/T) K_N(t,x) = J_1(t,x) + J_2(t,x),
$$
one can bound
$$
|E_\lambda| \lesssim \frac{1}{\lambda^2 - \| J_2 \|_{L^\infty}} \| \widehat{J_1} \|_{L^\infty}
$$
provided $\lambda^2 > \| J_2 \|_{L^\infty})$. The question becomes then: how can one optimally split $\chi(t/T) K_N(t,x)$ into $J_1 + J_2$, by making the norm of $J_2$ in $L^\infty$, and that of $J_1$ in $\widehat{L^\infty}$ small? Let
$$
S=
\big\{t\in[0,T]:
\sup_{x\in\T^2}\lvert \chi(t/T) K_N(t,x)\rvert
\geq \lambda^2\big\}
$$
and let $J_1 = \mathbf{1}_{S}(t)\chi(t/T) K_N(t,x)$ and $J_2=[1-\mathbf{1}_{S}(t)]\chi(t/T) K_N(t,x)$. The bound  $\lambda^2>\| J_2 \|_{L^\infty}$ is then trivial. But we compute
$$
\widehat{J_1}(k,\tau)
=  \chi \left( \frac{k_1}{N} \right) \chi \left( \frac{k_2}{N} \right) \int 
\phi \left(\frac{t}{T}\right)\mathbf{1}_{S}(t) e^{2\pi i t(\sum_{i,j} \alpha_{ij} k_i k_j-\tau) } \,dt
$$
and so
$$
\| \widehat{J_1} \|_{L^\infty} \lesssim 
\lvert S \rvert.
$$
The measure $\lvert S \rvert$ can be controlled using the bounds for $\sup_{x\in\T^2}\lvert K_N(t,x)\rvert$ discussed above.

The approach which was has been sketched is implemented in sections~\ref{section2} and~\ref{section3}:
\begin{itemize}
\item Section 2 is dedicated to establishing bounds for the Weyl sum $K_N$.
\item These bounds are then used in Section 3 to deduce Theorem~\ref{thmweyl} by the modified $TT^*$ argument.
\end{itemize}

\subsubsection{The counting argument}\label{sec:count_arg}
This is used to prove Theorem~\ref{thmp8}. If $p$ is an even integer, one deduces from \eqref{bigbird} that
$$
\| e^{it\widetilde \Delta} f \|_{L^p([0,T] \times \mathbb{T}^2)}^p = \sum_{\substack{k_1 \dots k_p\in \mathbb{Z}^2 \\ k_1-k_2+\dotsb-k_p=0}} \widehat{f}_{k_1} \overline{\widehat{f}_{k_2}}\dots \overline{\widehat{f}_{k_p}} \frac{1 - e^{2\pi i T \Omega(k_1,\dots,k_p)}}{2\pi i \Omega(k_1,\dots,k_p)},
$$
where, denoting $Q$ the quadratic form associated to $(\alpha_{ij})$,
$$
\Omega(k_1,\dots,k_p) = Q(k_1) - Q(k_2) + \dots - Q(k_p).
$$
After some manipulations, matters reduce to counting weighted solutions of Diophantine inequalities. In Section~\ref{section4}, the decoupling theory of Bourgain and Demeter~\cite{BD2} leads to the proof of the conjecture for $p \geq 16$ in a rather straightforward way. Alternatively, one can sum trivially over the odd numbered vectors $k_1,k_3,\dotsc$ and proceed by seeking a bound for the number of solutions to
\begin{align}\label{thought-it-was-a-bird-actually-a-squirrel}
\lvert Q(k_2)+Q(k_4)+\dotsb+Q(k_{p})-\beta\rvert &<\delta,
&k_2+k_4+\dotsb+k_{p}&=a,
&&k_i\in\Z^2\cap B(0,N),
\end{align}
which is uniform in $a,\beta$. To recover the conjecture it is necessary, in particular, to have an optimal bound in the case $\delta = \frac{1}{N}$.

There is previous work on bounds for the count of integer solutions to a generic quadratic inequality, where one wants uniformity in certain parameters~\cite{BN,GK,KY}. In those cases one in interested principally in lower bounds. In short, the strategy is to seek an asymptotic for the number of solutions which is valid unless the coefficients lie in a set with small measure. If the sum of this measure over all values of the parameters is finite, then by the Borel-Cantelli lemma the asymptotic holds for generic values of the coefficients.

It seems that this strategy does not work in our case, as we want uniformity in a rather large range of parameters and the space of coefficients has a relatively low dimension compared to the number of variables.

In Section~\ref{section5} we use an alternative approach to count solutions to \eqref{thought-it-was-a-bird-actually-a-squirrel}, proving the conjecture in the case $p=8$. We rely on the formula of Pall~\cite{PAL} on the number of solutions $(k_1,k_2,k_3,\ell_1,\ell_2,\ell_3)$ to the system
$$
k_1^2 + k_2^2 + k_3^2 = A, \qquad \ell_1^2 + \ell_2^2 + \ell_3^2 = B, \qquad k_1 \ell_1 + k_2 \ell_2 + k_3 \ell_3 = C.
$$
The ensuing analysis is delicate but repeatedly falls back on a simple estimate: for any $k\geq 2$ the number of solutions $x\in \Z^d\cap B(0,N)$ to some inequality $|f(x)-\nu|<\delta$ is at most $N_k^{1/k}$, where $N_k$ is the number of $k$-tuples $x_1,\dotsc,x_k\in \Z^d\cap B(0,N)$ such that $f(x_1)-f(x_i)=O(\delta)$ for each $i=2,\dotsc,k$. In particular $N_k$ is independent of the parameter $\nu$, allowing for uniform upper bounds.

Finally, Section~\ref{section6} contains the proof of the conjecture for $p=10$, also making use of Pall's formula, but this case turns out to be much simpler than $p=8$.

\subsection{Higher dimensions $d$}\label{sec:d>2}
One naturally asks whether the proof of Theorems~\ref{thmweyl} and~\ref{thmp8} can be extended to $d>2$.

We first consider Theorem~\ref{thmweyl}. 
To follow the strategy of section \ref{sec:weyl_arg} for $d>2$ we would define
$$
K_N^{(d)}(t,x) = \sum_{n \in \mathbb{Z}^d} \chi \left( \frac{n_1}{N} \right) \dotsm \chi \left( \frac{n_d}{N} \right) e^{2\pi i \left[x \cdot n+ t \sum_{i,j} \alpha_{ij} n_i n_j \right]},
$$
and we would then look for bounds on the quantities
\begin{equation}\label{rudolf}
S(T) = \sup_{x\in \T^d,\, t\in [T,2T]}\left\lvert
{K_N^{(d)}}(t,x)\right\rvert,
\qquad
I_s(T) = 
\int_{1}^{T} \sup_{x\in \T^d}\left\lvert
{K_N^{(d)}}(t,x)\right\rvert^{2s} \, dt
\qquad
(s\in\N).
\end{equation}
As a first bound on $I_s(T)$, it follows from the work of Guo and Zhang~\cite{GZ} that for generic $(\alpha_{ij})$ we have
\begin{equation}\label{jubjub}
I_{s}(T)
\lesssim
N^\epsilon T(N^{ds+d}+N^{2ds-d(d+1)}).
\end{equation}
{See Appendix~A for a brief account of the argument by which \eqref{jubjub} follows from the cited work. There we also sketch a strategy for replacing the $ds+d$ with $ds$.}

One might be able to imitate the proof of our Proposition~\ref{minkowski} to bound $I_s(T)$ and $S(T)$ from \eqref{rudolf}. To see how, let $E^d$ be the set of all $d\times d$ symmetric $(\alpha_{ij})$ such that $-2 \leq \alpha_{ij} \leq 2$, and with smallest absolute eigenvalue $|\lambda_1|>1$. To majorize $S(T)$ we would need to bound
\begin{multline*}
\mu_1(T,\vec{\gamma})=
\operatorname{measure}\big\{
(\alpha_{ij})\in E
:
\alpha_{i1}n_1^{(j)}+\dotsb+\alpha_{id}n_d^{(j)}
=t^{-1}m_i^{(j)}+O\Big(\frac{\gamma_j}{TN}\Big)
\\
\text{ for all } i,j=1,\dotsc,d
\text{ and some }
t\sim T,\,
n_i^{(j)}\lesssim \gamma_j N,\,
m_i^{(j)} \lesssim T\gamma_j N
\big\}
\end{multline*}
for appropriate $\frac{1}{N}\leq \gamma_i\leq 1$.
For $I_s(T)$ the analogous quantity is
\begin{multline*}
\mu_2(t)=
\operatorname{measure}\big\{
(\alpha_{ij})\in E
:
\alpha_{i1}n_1^{(j)}+\dotsb+\alpha_{id}n_d^{(j)}
=t^{-1}m_i^{(j)}+O\Big(\frac{1}{tN}\Big)
\\
\text{ for all } i,j=1,\dotsc,d
\text{ and some }
n_i^{(j)}\lesssim  N,\,
m_i^{(j)} \lesssim t N
\big\}.
\end{multline*}
We suggest that it may be possible to estimate $\mu_1$ and $\mu_2$ using the geometry of numbers.

Finally we consider extending the proof of Theorem~\ref{thmp8}, as outlined in section~\ref{sec:count_arg}, to the case $d>2$. The problem here is easier to describe. Given a symmetric integer matrix $(A_{ij})$, one requires a bound for the number of solutions $ y^{(1)},\dotsc,y^{(d+1)}) \in (\Z^d)^{d+1}$ with $ |y^{(k)}|\lesssim N$ to the system
$$
\sum_{k=1}^{d+1}y^{(k)}_iy^{(k)}_j=A_{ij}
\qquad\qquad(1\leq i\leq j \leq d)
$$
and one requires this bound to have explicit dependence on the $A_{ij}$. One possibility is to use the Siegel mass formula as in \cite{BD4}.

\subsection{Notations} 
\begin{itemize}
\item For $x$ a real number, $\| x \|$ denotes the smallest distance to an integer: $\| x \| = \min_{k \in \mathbb{Z}} |x - k|$.

\item We denote $A \lesssim B$ if there exists a universal constant such that $|A| \leq CB$; and $A\sim B$ if $A \lesssim B$ and $B \lesssim A$.

\item We denote $A \preceq B$ if, for any $\epsilon > 0$, there exists $C_\epsilon$ such that $|A| \leq C_\epsilon N^\epsilon B$.

\item {We follow the analytic (rather than the number analytic) convention in using the somewhat informal notation $A \ll B$ if there is a very small constant $c$ such that $A \leq c B$.}

\item Given a set $E$, its characteristic function $\mathbf{1}_E(x)$ equals $1$ if $x \in E$, and $0$ otherwise.

\item For $f(x)$ a function on the torus $\mathbb{T}$, its Fourier coefficients are
$$
\mbox{for $k\in \mathbb{Z}^2$}, \qquad \widehat{f}_k = \int_{\mathbb{T}^2} f(x) e^{-2\pi i k \cdot x} \,dx.
$$ 

\item For $f(t,x)$ a function on $\mathbb{R} \times \mathbb{T}$, its space-time Fourier transform is given by
$$
\mbox{for $(\tau,k) \in  \mathbb{R} \times \mathbb{Z}^2$}, \qquad \widehat{f}(\tau,k) = \int_\mathbb{R} \int_{\mathbb{T}^2} f(t,x) e^{-2\pi i k \cdot x -2\pi i t \tau} \,dx \,dt.
$$ 

\item {Finally, $\langle t \rangle = \sqrt{1+t^2}$.}
\end{itemize}

\subsection*{Acknowledgements} While working on this project, Y. D. was supported by the NSF grant DMS-1900251.

P. G. was supported by the NSF grant DMS-1501019, by the Simons collaborative grant on weak
turbulence, and by the Center for Stability, Instability and Turbulence (NYUAD).

L. G. was supported by a Simons Investigator Award.

S. M. was supported by the Engineering and Physical Sciences Research Council [EP/M507970/1], by the European Research Council under ERC grant agreement no. 670239, by the Fields Institute for Research in Mathematical Sciences, and by NSF-DMS 1363013.

\section{The conjecture}\label{conjecture-justification}

In order to formulate a conjecture, we examine a few particular functions $f$, and determine heuristically the ratio $\| e^{it{\widetilde{\Delta}}} f \|_{L^p ([0,T] \times \mathbb{T}^d)} / \| f \|_{L^2(\mathbb{T}^d)}$ which they yield.

\begin{enumerate}
	\item If $f =1$, we find $\| e^{it\widetilde{\Delta}} f \|_{L^p ([0,T] \times \mathbb{T}^d)} \sim T^{\frac{1}{p}}$.
	\item Choose $f(x) = N ^{\frac{d}{2}} \chi \left( Nx \right)$ for a smooth compactly supported function $\chi$. Then $f$ is normalized in $L^2$, with $\| f \|_{L^2(\mathbb{T}^d)} \sim 1$, and the transform $\hat{f}$ has rapid decay outside the ball $B(0,N)$. Furthermore, for $\frac{1}{N^2} \ll t \ll \frac{1}{N}$, $e^{it\widetilde \Delta} f$ nearly coincides with the solution of the Schr\"odinger equation on $\mathbb{R}^d$ with initial data $f$, which is mostly supported on $B(0,tN)$ and has size $\sim \frac{1}{(tN)^{d/2}}$; {this is because the solution to this PDE propagates at a group velocity $\lesssim N$, hence for times $\ll \frac{1}{N}$, the difference between the torus and the whole space is negligible}. Therefore
	$$
	\| e^{it{\widetilde{\Delta}}} f \|_{L^p ([0,1] \times \mathbb{T}^d)}^p
	\gtrsim
	\int_{\frac{1}{N^2}}^{\frac{1}{N}} \int_{B(0,tN)} \frac{dx\, dt}{(tN)^{dp/2}}
	\gtrsim
	N^{\frac{dp}{2} - (d+2)} \qquad \mbox{if $p> \frac{2(d+2)}{d}$}.
	$$
	\item We want to find the ``refocusing time" with the same $f$. In other words, how big should $t$ be taken in order that $e^{it\widetilde \Delta} f$ be similar to $f$ itself? Or what is an ``almost period" of the flow?
	
	On a rectangular torus all possible $e^{it \Delta} f$ are periodic with a common period, forcing the refocusing time to be $O(1)$; clearly this is not the typical behaviour of a nonpositively curved manifold.
	
	The solution $e^{it{\widetilde{\Delta}}} f$ can be expressed as
	$$
	e^{it{\widetilde{\Delta}}} f (y) = N^{-d/2} \sum_{k \in \mathbb{Z}^d} \widehat{\chi} \left( \frac{k}{N} \right) e^{2\pi i \left[y\cdot k+ t \sum \alpha_{ij} k_i k_j \right] }
	$$
	Heuristically, for $e^{it{\widetilde{\Delta}}} f$ to be similar to $f$, we need that
	$$
	\forall k \in \mathbb{Z}^d \quad \mbox{with $|k| \lesssim N$}, \qquad \| t \sum_{ij} \alpha_{ij} k_i k_j \| \ll 1,
	$$
	where $\| x \|$ is the smallest distance of $x$ to an integer {(this is simply because this makes the time dependent summand in the complex exponential close to a integer multiple of $2\pi$)}. This will be achieved if $\| t \alpha_{ij} \| \ll \frac{1}{N^2}$ for all $i,j$. Since $\alpha_{11} =1$, we will look for $t \in \mathbb{N}$; and since $\alpha_{ij} = \alpha_{ji}$, this leaves $\frac{d^2 + d -2}{2}$ independent coefficients.
	
	By Dirichlet's and Khinchin's approximation theorems, for generic $(\alpha_{ij})$, and for any $\epsilon>0$, there exists an integer $q$ such that $N^{d^2 + d -2-\epsilon} \lesssim q \lesssim N^{d^2 + d -2}$ and
	$$
	\forall i,j, \qquad \| q \alpha_{ij} \| < \frac{1}{N^2}.
	$$
	Choosing $t = q$, we find an almost period $\sim N^{d^2 + d -2}$ up to subpolynomial factors.
	
	It is natural to expect that, at each refocusing time, a contribution similar to that found in (2) will occur. This suggests that, for any $\epsilon>0$,
	$$
	\| e^{it{\widetilde{\Delta}}} f \|_{L^p ([0,T] \times \mathbb{T}^d)} \gtrsim N^{\frac{d}{2} - \frac{d+2}{p}} \left[ \frac{T}{N^{d^2 + d -2-\epsilon}} \right]^{\frac{1}{p}} > N^{\frac{d}{2} - \frac{d^2 + 2d}{p} - \epsilon} T^{\frac{1}{p}}.
	$$
	\item Compared to the rectangular case, we also have new competitors, namely $n$-dimensional data, where $n \in \{ 1,\dots,d-1 \}$. Here we choose $f(x)=N^{n/2}\chi(Nx_1,\dotsc,Nx_n)$ to depend only on (say) the first $n$ variables, replicating the behavior on an $n$-dimensional torus. By the above, this give examples for which
	$$
	\| e^{it{\widetilde{\Delta}}} f \|_{L^p ([0,T] \times \mathbb{T}^d)} \gtrsim N^{\frac{n}{2}-\frac{n^2 + 2n}{p}-\epsilon} T^{\frac{1}{p}} \| f \|_{L^2(\mathbb{T}^d)}.
	$$
	These may dominate the contribution from (3), because the refocusing time grows rapidly with the dimension.
\end{enumerate}
Overall, this gives the conjecture~\eqref{theconjecture}.

\begin{rem}\label{fulvous-tree-duck}
When $p$ is an even integer, parts (1), (3) and (4) of the heuristic above can be made rigorous as follows (the remaining part is more straightforward).

Let $p\in 2\N$ and consider the lower bound on the number of solutions of the quadratic  Parsell-Vinogradov system in Parsell, Prendiville and Wooley~\cite{PPW}. There, we learn that the number of solutions $(k^{(1)},\dots,k^{(p)}) \in (\mathbb{Z}^d)^p$ of
$$
\forall \ell,m,n,\qquad \sum_{j = 1}^p (-1)^j k^{(j)}_\ell = 0 \quad \mbox{and} \quad \sum_{j = 1}^p (-1)^j k^{(j)}_m k^{(j)}_n = 0,
$$
which furthermore satisfy $|k^{(j)}| < N$, is
$$
\gtrsim N^{\frac{pd}{2}} + \sum_{j=1}^d N^{(p-1) j + d - j (j+2)} \geq N^{\frac{pd}{2}} + N^{pd-d(d+2)}
$$
(this is Theorem 1.2 in that paper; the notations there are related to ours by $2s=p$, $k=2$, while $d$ remains the same). Defining $f$ by its Fourier coefficients $\widehat{f}_k =  \mathbf{1}_{|k| < N}$,  the above bound implies that 
$$
\| e^{it\widetilde{\Delta}} f \|_{L^p([0,T]\times \mathbb{T}^d)}
\gtrsim (N^{\frac{d}{2}}+N^{d - \frac{d(d+2)}{p}}) T^{\frac{1}{p}} 
> \max(1,N^{\frac{d}{2} - \frac{d(d+2)}{p}}) T^{\frac{1}{p}} \| f \|_{L^2(\mathbb{T}^d)}
$$
for generic $\alpha_{ij}$ and large $T$. Considering lower dimensional examples, we find that
$$
\mbox{as $T \to \infty$}, \qquad \sup_f \frac{ \| e^{it {\widetilde{\Delta}}} f \|_{L^p ([0,T]  \times \mathbb{T}^d)} } {\| f \|_{L^2(\mathbb{T}^d)} } \gtrsim T^{1/p} \sum_{n=0}^d N^{\frac{n}{2} - \frac{n^2 + 2n}{p}}.
$$
This is consistent with the conjecture which was heuristically derived above.
\end{rem}

\section{Weyl sum estimates}

\label{section2}

\subsection{Statement of the results}

For $\chi$ a smooth, nonnegative function on $\mathbb{R}$, equal to $1$ on $B(0,\frac{1}{2})$ and $0$ on $B(0,1)^\complement$, we define the regularized fundamental solution
$$
K_N(t,x) = \sum_{n \in \mathbb{Z}^2} \chi \left( \frac{n_1}{N} \right) \chi \left( \frac{n_2}{N} \right) e^{2\pi i \left[x \cdot n+ t \sum_{i,j} \alpha_{ij} n_i n_j \right]}.
$$

We state without proof the following classical bound for short time, which does not use the genericity of the matrix $(\alpha_{ij})$.

\begin{lem} \label{dispersive}
If $|t| \lesssim \frac{1}{N}$, then 
$$
\forall x \in \mathbb{T}^2, \qquad |K_N(t,x)| \lesssim \min  \left( N^2 ,\frac{1}{t} \right).
$$
\end{lem}

{\begin{proof} This bound can be proved first on $\mathbb{R}^2$ by stationary phase, and then transferred to the torus by Poisson summation.\end{proof}}

For larger time, we obtain the following bounds.

\begin{prop}\label{minkowski} Consider a generic positive symmetric matrix $(\alpha_{ij})_{1 \leq i,j \leq 2}$, with eigenvalues $\lambda_1 < \lambda_2$, which satisfies
	\begin{equation}
	\label{annashummingbird}
	\forall \; i,j, \quad |\alpha_{ij}| \leq 2, \quad \mbox{and} \quad \lambda_1 > 1.
	\end{equation}
	Then
	\begin{itemize}
		\item[(i)]
		For any $t > \frac{1}{N}$ (and {recalling that} $\langle t \rangle = \sqrt{1+t^2}$)
		$$
		\sup_{x \in \mathbb{T}^2} |K_N(t,x)| \preceq N^{4/3} \langle t \rangle^{1/6}.
		$$
		\item[(ii)] For any $T \geq 1$, 
		$$
		\int_1^T \sup_{x \in \mathbb{T}^2} |K_N(t,x)|^4 \, dt \preceq N^4 T.
		$$
	\end{itemize}
\end{prop}

\subsection{Proof of $(i)$ in Proposition~\ref{minkowski}}

The first step is to apply Weyl { differencing} to obtain the following lemma, whose proof we postpone for the moment.

\begin{lem} \label{merganser1} Denoting $L_j(n) = \sum \alpha_{ij} n_i$, for any $x \in \mathbb{T}^2$ and $t \in \mathbb{R}$,
$$
|K_N(t,x)|^2  \lesssim \sum_{r_1=-2N}^{2N} \sum_{r_2=-2N}^{2N} \prod_{j=1}^{2} \min\left( N, \frac{1}{\| 2 t L_j(r) \|} \right).
$$
\end{lem}

By the above,
$$
|K_N(t,x)|^2  \lesssim \sum_{m_1,m_2} \# E^N_{m_1,m_2} \frac{N^2}{\langle m_1 \rangle \langle m_2 \rangle}
$$
where
$$
E_{m_1,m_2}^N = \{ n \in [-2N,2N]^2 \; \mbox{s.t.} \; \frac{m_j-1}{N} \leq \| tL_j(n) \| \leq \frac{m_j}{N} \; \mbox{for $j=1,2$} \}
$$
But since any $n,n'$ in $E_{m_1,m_2}^N$ are such that $|n-n'| < 4N$, and $\| t L_j(n-n') \| < \frac{2}{N}$, we find that
\begin{equation}\label{bowerbird}
|K_N(t,x)|^2 \preceq N^2 \left[ \# \{ n \in [-4N,4N]^2 \;\; \mbox{s.t.} \;\; \| t L_j(n) \| < \frac{2}{N} \} + 1 \right].
\end{equation}

We now apply Lemma 3 from~\cite{Davenport58} to get
$$
|K_N(t,x)|^2 \preceq \frac{N^2}{M_1^t  M_2^t},
$$
where $(M_i^t)$ are the Minkowski minima for the norm on $\mathbb{R}^{4}$ given by
$$
\mbox{if} \; (n,m) \in \mathbb{Z}^{4}, \quad F(n,m) = \max \left( \left|\frac{n_i}{N}\right|, N |t L_i(n)-m_i| \right).
$$
Recall that $M_k^t$ is the smallest real number $r$ such that the set $\{ (n,m) \in \mathbb{Z}^4 \; \mbox \; \mbox{such that} \; F(n,m) \leq r \}$ contains $k$ independent vectors.

\bigskip

Below, $t$ is such that $t \sim T \in 2^\mathbb{Z}$. We introduce dyadic scales $\beta$ and $\gamma$. By definition, $M^t_1 \leq \beta$ and $M_2^t \leq \frac{\gamma}{ \beta}$ if there exists $n,m,n',m' \in \mathbb{Z}^{2}$ such that 
\begin{subequations}
\begin{align}
& (n,m) \neq (0,0)\; \mbox{and} \; (n',m') \;  \mbox{not colinear to}\; (n,m)\label{baldeagle0}  \\
&|n_1|, |n_2| \leq \beta N \label{baldeagle1}  \\
&|n'_1|, |n'_2| \leq \frac{\gamma}{\beta}  N \label{baldeagle2} \\
&|m| \lesssim T \beta N \label{baldeagle3} \\
&|m'| \lesssim T \frac{\gamma}{\beta} N \label{baldeagle4}
\end{align}
\end{subequations}
and furthermore
\begin{subequations}
\begin{align}
& \left| t (\alpha_{11} n_1 + \alpha_{12} n_2 ) - m_1 \right| \leq \frac{\beta}{N} \label{goldfinch1} \\
& \left| t (\alpha_{12} n_1 + \alpha_{22} n_2 ) - m_2 \right| \leq \frac{\beta}{N} \label{goldfinch2} \\
& \left| t (\alpha_{11} n_1' + \alpha_{12} n_2' ) - m_1' \right| \leq \frac{\gamma}{\beta N} \label{goldfinch3} \\
& \left| t (\alpha_{12} n_1' + \alpha_{22} n_2') - m_2' \right| \leq \frac{\gamma}{\beta N} \label{goldfinch4}.
\end{align}
\end{subequations}
We are interested in the regime where $\frac{1}{N} < \beta < \sqrt \gamma \ll 1$ (the first inequality is needed for $n$ to be nonzero; the second one expresses the fact that $M_1 \leq M_2$; and the third one is related to the kernel bound we want to prove). Furthermore, there are no solutions of the above unless
\begin{equation}
\label{mallard}
\beta \gtrsim \frac{1}{NT}.
\end{equation}
Indeed, if $T \beta N \ll1$, we find $m=0$ by~\eqref{baldeagle3}, and then, by~\eqref{goldfinch1} and~\eqref{goldfinch2}, $|n| \lesssim \frac{\beta}{NT} \lesssim T\beta N \ll1$. Thus $(n,m)=0$, which is excluded.

Let
$$
E_{\beta,\gamma,N,T}^{n,n'} = \{ (\alpha_{ij})_{1\leq i,j \leq 2} \; \mbox{such that} ~\eqref{annashummingbird},~\mbox{and}~\eqref{goldfinch1}~\mbox{to} ~\eqref{goldfinch4} ~\mbox{hold for some $m,m' \in \mathbb{Z}^2$ and $t \sim T$} \}
$$
{(notice that any $m,m'$ appearing in the above definition automatically satisfy~\eqref{baldeagle3} and~\eqref{baldeagle4} as soon as $n,n'$ satisfy~\eqref{baldeagle1} and~\eqref{baldeagle2}).}

The next lemma provides an upper bound for the size of $E_{\beta,\gamma,N,T}^{n,n'}$; we postpone its proof for the moment.

\begin{lem} \label{merganser2} 
\begin{itemize}
\item[(i)] If $n\times n' = 0$, 
$$E_{\beta,\gamma,N,T}^{n,n'}= \emptyset.$$
\item[(ii)] If $n \times n' \neq 0$, assuming that $n_2' \neq 0$,
$$
|E_{\beta,\gamma,N,T}^{n,n'}| \preceq \frac{\gamma^2 N}{|n \times n'|} \left( 1 + \frac{T \gamma N \gcd(n_2,n_2')}{\beta |n_2'|} \right) \min \left( \frac{\beta}{|n|}, \frac{ \gamma}{\beta|n'|} \right).
$$
\end{itemize}
\end{lem}

Our aim is to estimate $\sum_{n,n'} | E_{\beta,\gamma,N,T}^{n,n'} | $. The cases where one of $n_1,n_2,n_1',n_2'$ is zero is easier to deal with, so we will omit them and assume below that $n_1n_2n_1'n_2'\neq 0$.

We split this sum into two pieces: by~Lemma~\ref{merganser2},
\begin{align*}
& \sum_{n,n'} | E_{\beta,\gamma,N,T}^{n,n'} | \\
& \preceq \sum_{n \times n' \neq 0} \frac{\gamma^2 N}{|n \times n'|} \left( 1 + \frac{T \gamma N \gcd(n_2,n_2')}{\beta |n_2'|} \right) \min \left( \frac{\beta }{|n|}, \frac{ \gamma}{\beta|n'|} \right)\\
& = \sum_{\substack{n \times n' \neq 0\\ \beta |n_2'| < T \gamma N \gcd(n_2,n_2')}} + \sum_{\substack{n \times n' \neq 0\\ \beta |n_2'| \geq T \gamma N \gcd(n_2,n_2')}} \dots \\
&  = \Sigma_1 + \Sigma_2
\end{align*}
(the summations above are always understood under the conditions~\eqref{baldeagle1} and~\eqref{baldeagle2})

\bigskip

\noindent \underline{Bound for $\Sigma_1$} It can be controlled by
$$
\gamma^3 TN^2 \sum_{n,n'} \frac{\gcd(n_2,n_2')}{|n_2'||n \times n'| |n_1|}.
$$
Now we have
$$
\sum_{n,n'} \frac{\gcd(n_2,n_2')}{|n_2'||n \times n'| |n_1|} = \sum_\lambda \frac{1}{\lambda} \sum_{\substack{n_2 = \lambda \nu_2 \\ n_2' = \lambda \nu_2' \\ \gcd(\nu_2,\nu'_2)=1}} \frac{1}{|n_1| |\nu'_2| |n_1 \nu'_2 - n_1' \nu_2 |} \preceq 1,
$$ due to divisor bounds,
which gives
$$
\Sigma_1 \preceq \gamma^3 T N^2.
$$

\bigskip

\noindent \underline{Bound for $\Sigma_2$} It is less than
$$
\sum_{n,n'}
\frac{\gamma^3 N}{\beta |n \times n'| |n_2'|}.
$$
We have
\begin{align*}
\sum_{n,n'}
\frac{1}{|n \times n'| |n_2'|} = \sum_{n_1,n_2'} \frac{1}{|n_2'|} \sum_{n_1',n_2} \frac{1}{|n_1 n_2' - n_2 n_1'|} \preceq \sum_{n_1,n_2'} \frac{1}{|n_2'|} \preceq \beta N.
\end{align*}
In other words,
$$
\Sigma_2 \preceq \gamma^3 N^2.
$$

\bigskip

We can now conclude the proof of Proposition~\eqref{minkowski}: we just showed that
$$
\sum_{n,n'} | E_{\beta,\gamma,N,T}^{n,n'} | \preceq \gamma^3 N^2 (1+T).
$$
Since $\frac{1}{N} < \beta < 1$, this implies that
$$
\sum_\beta \sum_{n,n'} | E_{\beta,\gamma,N,T}^{n,n'} | \preceq \gamma^3 N^2 (1+T).
$$
We now distinguish between different regimes:
\begin{itemize}
\item If $T>1$, set $\gamma = T^{-1/3-\delta} N^{-2/3-\delta}$, for an arbitrarily small $\delta$.
\item If $ T < 1$, set $\gamma = T^{-\delta} N^{-2/3-\delta}$, for an arbitrarily small $\delta$.
\end{itemize}
With these choices for $\gamma$ (which ensure that $\gamma \ll 1$), we obtain
$$
\sum_{n,n'}  \sum_{\beta,N,T} | E_{\beta,\gamma,N,T}^{n,n'} | < \infty.
$$
Therefore, by Borel-Cantelli, we learn that almost surely in $(\alpha_{ij})$,
$$
M_1^t M_2^t \succeq \gamma(t,N).
$$
This gives
$$
\forall t > \frac{1}{N} \; \mbox{and} \; x \in \mathbb{T}^d, \qquad |K_N(t,x)| \preceq N^{4/3} \langle T \rangle^{1/6}.
$$

\subsection{Proof of $(ii)$ in Proposition~\ref{minkowski}}

Just like in the proof of $(i)$, the first step is to appeal to Lemma~\ref{merganser1} to obtain that
$$
|K_N(t,x)|^2  \lesssim \sum_{n_1=-2N}^{2N} \sum_{n_2=-2N}^{2N} \prod_{j=1}^{2} \min\left( N, \frac{1}{\| t L_j(n) \|} \right).
$$
Denoting $P(t)$ for the above right-hand side
$$
P(t) = \sum_{n_1,n_2}  \prod_{j=1}^{2} \min\left( N, \frac{1}{\| t L_j(n) \|} \right),
$$
observe that
\begin{equation}
\label{greypartridge}
\begin{split}
\iint |P(t)|^2 \,d\alpha_{11} \,d\alpha_{22} & \lesssim
\sum_{\substack{n_1,n'_1 \\ n_2,n'_2}} \int \frac{1}{\| t(\alpha_{11} n_1 + \alpha_{12} n_2) \| + \frac{1}{N}} \frac{1}{\| t(\alpha_{11} n'_1 + \alpha_{12} n'_2) \| + \frac{1}{N}} \,d\alpha_{11} \\ 
& \qquad \int \frac{1}{\| t(\alpha_{12} n_1 + \alpha_{22} n_2) \| + \frac{1}{N}} \frac{1}{\| t(\alpha_{12} n'_1 + \alpha_{22} n'_2) \| + \frac{1}{N}} \,d\alpha_{22}\\
& = \sum_{\substack{n_1,n'_1 \\ n_2,n'_2}} S_1(t) S_2(t).
\end{split}
\end{equation}
{The case when either of $n_1,n_1',n_2,n_2'$ is zero can be dealt with in a similar fashion to the general case; therefore, we assume in the following that none of $n_1,n_1',n_2,n_2'$ vanishes.} Since $S_1$ and $S_2$ are symmetrical, we focuse on the former and write
$$
\left\{
\begin{array}{l}
n_1 = k p_1 \\
n_1' = k p_1'
\end{array}
\right.
\quad \mbox{with} \quad \operatorname{gcd}(p_1,p'_1) = 1,
$$
we claim that (uniformly in $n_2,n_2',\alpha_{12}$ and for $t \geq 1$)
\begin{equation}
\label{peregrinefalcon}
S_1(t) \preceq \frac{N}{ |p_1|+|p_1'|}.
\end{equation}
Coming back to~\eqref{greypartridge}, this implies that
$$
\iint |P(t)|^2 \,d\alpha_{11} \,d\alpha_{22} \preceq  \bigg(\sum_{p_1,p_1'}\frac{N}{|p_1|+|p_1'|}\sum_{k:|kp_1|\leq N,|kp_1'|\leq N}1\bigg)^2\lesssim N^4\bigg(\sum_{p_1,p_1'}\frac{1}{(|p_1|+|p_1'|)^2}\bigg)^2\preceq N^4,
$$
which in turn implies that
$$
\iint \int_1^T \sup_x |K_N(t,x)|^4 \, dt \,d\alpha_{11} \,d\alpha_{22} \preceq N^4 T,
$$
from which the desired result follows by a Borel-Cantelli type argument.

\bigskip

There remains to prove~\eqref{peregrinefalcon}. From now on, $n,n',t,\alpha_{12}$ are fixed. First decompose the arguments into integer and fractional part
$$
\left\{
\begin{array}{l}
t \alpha_{11} n_1 + t \alpha_{12} n_2 = \nu + \delta \\
t \alpha_{11} n'_1 + t \alpha_{12} n'_2 = \nu' + \delta'
\end{array}
\right.
\quad
\mbox{where}
\quad
\nu,\nu' \in \mathbb{Z} \quad \mbox{and} \quad -\frac{1}{2} \leq \delta,\delta' < \frac{1}{2}.
$$
Furthermore, define $\epsilon,\epsilon' \in 2^\mathbb{Z} \cap [\frac{1}{N},\frac{1}{2}]$ by $|\delta| \sim \epsilon$ if $|\delta| > \frac{1}{N}$; and $\epsilon \sim \frac{1}{N}$ if $|\delta| < \frac{1}{N}$ so that
$$
S_1(t) \lesssim \sum_{n_1,n_1'} \int \frac{1}{\epsilon \epsilon'} \,d\alpha_{11}.
$$
We now argue as follows:
\begin{itemize}
	\item For $\epsilon,\epsilon'$ fixed, $\nu$ and $\nu'$ satisfy
	$$
	| \nu p_1' - \nu' p_1 - C| < \epsilon |p_1'| + \epsilon' |p_1|,
	$$
	where $C$ is a constant dedending on $n$, $n'$, $\alpha_{12}$ and $t$. The number of $(\nu,\nu')$ satisfying this constraint is
	$$
	\preceq tk ( \epsilon |p_1'| + \epsilon' |p_1| + 1).
	$$
	Indeed, for each fixed value of $\nu p_1' - \nu' p_1$, $\nu$ is determined modulo $p_1$, and it ranges in an interval of size $\sim t|n_1|$, leaving $\lesssim \frac{t|n_1|}{|p_1|}=tk$ possibilities.  Finally, the number of possible values for $\nu p_1' - \nu' p_1$ is $\lesssim \epsilon |p_1'| + \epsilon' |p_1| + 1$.
	\item Furthermore, given $\nu,\nu',\epsilon,\epsilon'$, the measure of possible $\alpha_{11}$ is 
	$$
	\lesssim \min \left( \frac{\epsilon}{t|n_1|} , \frac{\epsilon'}{t|n_1'|} \right).
	$$
\end{itemize}
This leads to the bound
\begin{align*}
|S_1(t)| \preceq \sum_{\epsilon,\epsilon'} \frac{1}{\epsilon \epsilon'} tk ( \epsilon |p_1'| + \epsilon' |p_1| + 1) \min \left( \frac{\epsilon}{t |n_1|} , \frac{\epsilon'}{t|n_1'|} \right) \lesssim \sum_{\epsilon,\epsilon'} 1 + \frac{1}{\epsilon' |p_1|} \preceq \frac{N}{|p_1|},
\end{align*} the bound involving $p_1'$ following by symmetry.

\subsection{Proof of Lemma~\ref{merganser1}}
Squaring $K_N$, and using that $(\alpha_{ij})$ is symmetrical, gives
$$
|K_N(t,x)|^2 = \sum_{n \in \mathbb{Z}^2} \sum_{n' \in \mathbb{Z}^2}\chi \left( \frac{n_1}{N} \right) \chi \left( \frac{n_2}{N} \right) \chi \left( \frac{n_1'}{N} \right) \chi \left( \frac{n_2'}{N} \right) e^{2\pi i \left[ \sum x_i (n_i-n_i') + t \sum_{i,j} \alpha_{ij} (n_i -n_i') (n_j + n_j') \right]}.
$$
Changing variables to $r = n-n'$, $s = n+n'$, this becomes
$$
|K_N(t,x)|^2 = \sum_{r \in \mathbb{Z}^2} e^{2\pi i \sum r_i x_i} \sum_{s \in \mathbb{Z}^2}^* \chi \left( \frac{r_1+s_1}{2N} \right) \chi \left( \frac{r_2 + s_2}{2N} \right) \chi \left( \frac{s_1 - r_1}{2N} \right) \chi \left( \frac{s_2 - r_2}{2N} \right) e^{2\pi i  t \sum_{i,j} \alpha_{ij} r_i s_j},
$$
where $\sum_s^*$ means that the sum is restricted to these $s$ such that, for all $i$, $s_i$ and $r_i$ have the same parity. The second sum above factors into a sum over $s_1$ times a sum over $s_2$; we focus on the sum over $s_1$, which reads
$$
\sum_{s_1}^* \chi \left( \frac{r_1+s_1}{2N} \right) \chi \left( \frac{s_1 - r_1}{2N} \right) e^{2\pi i t s_1 \sum_{j} \alpha_{1j} r_j}.
$$
By Abel summation, the modulus of this sum is
$$
\dots \lesssim \min \left( N, \frac{1}{ \left\|2t \sum \alpha_{1j} r_j \right\|} \right).
$$
Overall, we find
$$
|K_N(t,x)|^2  \lesssim \sum_{r_1=-2N}^{2N} \sum_{r_2=-2N}^{2N} \prod_{j=1}^{2} \min\left( N, \frac{1}{\|2 t L_j(r) \|} \right),
$$
where $L_j(r) = \sum \alpha_{ij} r_i$. 

\subsection{Proof of Lemma~\ref{merganser2}} 
\underline{(i) The case $n \times n' = 0$.} We argue by contradiction and start by assuming that $E_{\beta,\gamma,N,T}^{n,n'}$ is not empty. Since $n$ and $n'$ are aligned, they can be written
$$
n = kp, \qquad n' = kp',
$$
where $k \in \mathbb{Z}^2$ has relatively prime coordinates, and $p,p' \in \mathbb{Z}$. The inequalities~\eqref{goldfinch1} to \eqref{goldfinch4} become
\bigskip
\begin{align*}
& \left| t (\alpha_{11} k_1 + \alpha_{12} k_2 ) - \frac{m_1}{p} \right| \leq \frac{\beta}{N|p|} \\
& \left| t (\alpha_{12} k_1 + \alpha_{22} k_2 ) - \frac{m_2}{p} \right| \leq \frac{\beta}{N|p|} \\
& \left| t (\alpha_{11} k_1 + \alpha_{12} k_2 ) - \frac{m_1'}{p'} \right| \leq \frac{\gamma}{\beta N|p'|} \\
& \left| t (\alpha_{12} k_1 + \alpha_{22} k_2) - \frac{m_2'}{p'} \right| \leq \frac{\gamma}{\beta N|p'|} .
\end{align*}
This implies that, on the one hand,
$$
\left| \begin{pmatrix} m_1/p \\ m_2 / p  \end{pmatrix} - \begin{pmatrix} m_1'/p' \\ m_2'/p' \end{pmatrix} \right| \lesssim \frac{\beta}{N|p|} + \frac{\gamma}{\beta N|p'|}.
$$
But, on the other hand, by definition of the Minkowski minima, $(m,n)$ and $(m',n')$ cannot be colinear, therefore
$$
\left| \begin{pmatrix} m_1/p \\ m_2 / p  \end{pmatrix} - \begin{pmatrix} m_1'/p' \\ m_2'/p' \end{pmatrix} \right| \geq \frac{1}{|pp'|}.
$$
The two above inequalities imply that
$$
1 \lesssim \frac{\beta|p'| }{N} + \frac{\gamma |p|}{\beta N}.
$$
But this leads to a contradiction since
$$
\frac{\beta|p'| }{N} + \frac{\gamma |p|}{\beta N} \leq \frac{\beta|n'| }{N} + \frac{\gamma |n|}{\beta N} \lesssim \gamma \ll 1.
$$

\bigskip

\noindent \underline{(ii) The case $n \times n' \neq 0$.} 
In order to estimate the size of $E_{\beta,\gamma,N,T}^{n,n'}$, we proceed in three steps.
\begin{itemize}
\item First freezing $t$, we note that the system~\eqref{goldfinch1} to ~\eqref{goldfinch4} is overdetermined in $(\alpha_{ij})$, which results in a compatibility condition on $m,m'$. To derive this compatibility condition, observe that {solving for $\alpha_{12}$ }by~\eqref{goldfinch1} and~\eqref{goldfinch3}  or~\eqref{goldfinch2} and~\eqref{goldfinch4} gives, respectively,
\begin{align*}
& \alpha_{12} = \frac{1}{t(n \times n')} (-n_1' m_1 + n_1 m_1') + O \left( \frac{\gamma}{T |n \times n'|} \right) \\
& \alpha_{12} = \frac{1}{t(n \times n')} (n_2' m_2 - n_2 m_2') + O \left( \frac{\gamma}{T |n \times n'|} \right).
\end{align*}
Since $\gamma \ll1$, these two equalities can only hold if
\begin{equation}
\label{bluejay}
- n_1' m_1 + n_1 m_1' = n_2' m_2 - n_2 m_2'.
\end{equation}
To estimate the number of $m,m'$ staisfying the above, note that, on the one hand, 
by~\eqref{baldeagle3},~\eqref{baldeagle4}, and~\eqref{mallard}, the number of possible choices for $m_1$ and $m_1'$ is $\lesssim \gamma(TN)^2$. On the other hand, the number of solutions $(m_2, m'_2)$ of~\eqref{bluejay} for $n,n',m_1,m_1'$ fixed is
$$
\lesssim 1 + \frac{T \gamma N \gcd(n_2,n_2')}{\beta |n_2'|}.
$$
Overall, the number of solutions of~\eqref{bluejay} in $(m,m')$ for $(n,n')$ fixed is thus
\begin{equation*}
\lesssim \gamma(TN)^2 \left( 1 + \frac{T \gamma N \gcd(n_2,n_2')}{\beta |n_2'|} \right).
\end{equation*}.
\item With $t$ still frozen, and now $m$ and $m'$ fixed, we use~\eqref{goldfinch1} and~\eqref{goldfinch3} to solve for $\alpha_{11}$ and $\alpha_{12}$. This gives a set of measure $\sim \frac{\gamma}{N^2 T^2} \frac{1}{|n \times n'|}$ . Next use~\eqref{goldfinch2} to solve for $\alpha_{22}$. This gives a set of measure $\sim \frac{\beta}{NT} \frac{1}{|n_2|}$. Symmetrically, one could use~\eqref{goldfinch4} to solve for $\alpha_{22}$, giving a set of measure $\sim \frac{\gamma}{\beta NT} \frac{1}{|n_2'|}$. By symmetry between the first and second coordinates, we get a bound
$\lesssim \frac{\gamma}{N^3T^3}\frac{1}{|n \times n'|} \min \left( \frac{\beta}{|n|}, \frac{ \gamma}{\beta|n'|} \right)$.

\item Now observe that if $t$ changes by an amount $dt \ll \frac{1}{N^2}$, the inequalities \eqref{goldfinch1} to \eqref{goldfinch4} need only be modified by a constant factor on the right-hand side. Since we want to cover the range $t \sim T$, it suffices to consider a number $O(N^2 T)$ of discrete times.
\end{itemize}

Overall, we find
$$
|E_{\beta,\gamma,N,T}^{n,n'}| \preceq \frac{\gamma^2 N}{|n \times n'|} \left( 1 + \frac{T \gamma N \gcd(n_2,n_2')}{\beta |n_2'|} \right) \min \left( \frac{\beta}{|n|}, \frac{ \gamma}{\beta|n'|} \right).
$$

\section{From Weyl sum estimates to Strichartz estimates}

\label{section3}

This section is dedicated to the proof of Theorem~\ref{thmweyl}.

\bigskip

\noindent
\underline{Step 1: decompositions of the kernel} Let $\phi$ be a smooth, real, non-negative function supported on $B(0,2)$ such that $\phi > 1$ on $B(0,1)$ and $\widehat{\phi} \geq 0$. For a number $A \in (0,\frac{1}{N})$ to be fixed later, decompose $\phi \left(\frac{t}{T}\right)K_N(t,x) $ into
\begin{align*}
\phi \left(\frac{t}{T}\right) K_N(t,x) = 
\underbrace{ \phi \left(\frac{t}{T}\right)\chi \left( \frac{t}{A} \right)K_N(t,x)}_{\displaystyle J_1(t,x)} 
& + \underbrace{\phi \left(\frac{t}{T}\right) \chi({Nt/2}) \left[ 1- \chi \left( \frac{t}{A}  \right)\right]K_N(t,x)}_{\displaystyle J_2(t,x)} \\
& + \underbrace{\phi \left(\frac{t}{T}\right) \left[ 1 - \chi({Nt/2}) \right]K_N(t,x)}_{\displaystyle J_3(t,x)}.
\end{align*}
By Lemma~\ref{dispersive} and Proposition~\ref{minkowski}, {and using the time support of $J_1$}, we find
\begin{align*}
& \| \widehat{J_1} \|_{L^\infty} \lesssim A \\
& \| J_2 \|_{L^\infty} \lesssim \frac{1}{A} \\
& \| J_3 \|_{L^\infty} \preceq N^{4/3} T^{1/6},
\end{align*}
{where the third bound holds for generic $\alpha_{ij}$.} Introducing the set
$$
S_\mu = \{ t \;\; \mbox{s.t.} \; t \geq 1 \; \mbox{and} \; \sup_x K_N(t,x) > \mu\},
$$
we learn from Proposition~\ref{minkowski} and Chebyshev's inequality that 
$$
|S_\mu| \preceq \frac{N^4 T}{\mu^4}.
$$
The above decomposition can be refined by letting
$$
J_3 = J_3 \mathbf{1}_{S_\mu} + J_3 (1-\mathbf{1}_{S_\mu}) = J_3' + J_3''.
$$
The bounds are now
\begin{align*}
& \| \widehat{J_3'} \|_{L^\infty} \preceq \frac{N^4T}{\mu^4} \\
& \| J_3'' \|_{L^\infty} \leq \mu,
\end{align*}
provided $\mu > N^{4/3}$ (in order to be able to estimate $\| J_3'' \|_{L^\infty([0,1]\times \mathbb{T}^2)}$).

\bigskip \noindent
\underline{Step 2: level set estimates.} We essentially follow the argument in Bourgain~\cite{B1}. Start with $f \in L^2 (\mathbb{T}^2)$ supported in Fourier on $B(0,N/2)$ and of norm 1: $\|f\|_{L^2(\mathbb{T}^2)} = 1$. Setting $F = e^{it \widetilde \Delta} f$, we want to estimate the size of 
$$
E_\lambda = \{ (x,t) \in \mathbb{T}^2 \times [0,T] \;\; \mbox{s.t.}  \;\; |F(x,t)| > \lambda \}.
$$
Set $\widetilde{F} = \frac{F}{|F|} {{1}\mkern-10mu{1}}_{E_\lambda}$ and $Q(k) = \sum_{ij} \alpha_{ij} k_i k_j$. We can bound, using successively Plancherel's theorem, the Cauchy Schwarz inequality, and again Plancherel's theorem,
\begin{align*}
\lambda^2 &|E_\lambda|^2  \lesssim \left[ \int_{\mathbb{T}^2 \times \mathbb{R}} \widetilde{F}(x,t) \overline{F(x,t) \phi \left( \frac{t}{T} \right)} \,\mathrm{d}x\,\mathrm{d}t \right]^2 \\
& = \left[ \sum_{k} \int \widehat{ \widetilde{F}}(\tau,k) \overline{T \widehat{f}_k \widehat{\phi}(T(\tau + Q(k)))} \chi \left( \frac{k_1}{N} \right)^{1/2}  \chi \left( \frac{k_2}{N} \right)^{1/2} \,\mathrm{d}\tau \right]^2 \\
& \leq \left[ \sum_{k} \int_{\mathbb{R}} \left| \widehat{\widetilde{F}} (\tau,k) \right|^2 T \widehat{\phi}(T(\tau + Q(k))) \chi \left( \frac{k_1}{N} \right)  \chi \left( \frac{k_2}{N} \right)  \,\mathrm{d}\tau \right] \left[ \sum_k |\widehat{f}_k|^2 \int T \widehat{\phi}(T(\tau + Q(k))) \,\mathrm{d}\tau \right] \\
& \lesssim \sum_k \int \left| \widehat{\widetilde{F}} (\tau,k) \right|^2 T \widehat{\phi} (T(\tau + Q(k)))\chi \left( \frac{k_1}{N} \right)  \chi \left( \frac{k_2}{N} \right) \,\mathrm{d}\tau\\
& = \int \left[ \left( K_N \phi \left(\frac{.}{T}\right) \right) * \widetilde{F} \right] (t,x) \overline{\widetilde{F}(t,x)} \,\mathrm{d}x\,\mathrm{d}t.
\end{align*}
Now using the first decomposition of Step 1,
\begin{align*}
\lambda^2 |E_\lambda|^2 & \lesssim \left< (J_1 + J _2 + J_3) * \widetilde{F}\,,\,\widetilde{F} \right> \\
& \lesssim \|\widehat{J_1}\|_{L^\infty} \|\widetilde{F}\|_{L^2}^2 + \left(\|J_2\|_{L^\infty}  + \|J_3\|_{L^\infty} \right) \|\widetilde{F}\|_{L^1}^2 \\
& \preceq A |E_\lambda| + \left( \frac{1}{A} +  N^{4/3} T^{1/6} \right) |E_\lambda|^2.
\end{align*}
Summarizing, we get if $A < \frac{1}{N}$
\begin{equation}
\label{levelsetestimate}
\lambda^2 |E_\lambda|^2 \preceq A |E_\lambda| + \left( \frac{1}{A} + N^{4/3} T^{1/6} \right) |E_\lambda|^2
\end{equation}
If we now resort to the refined decomposition of Step 1, we are led to
\begin{equation}
\label{levelsetestimate2}
\lambda^2 |E_\lambda|^2 \preceq \left( A + \frac{N^4 T}{\mu^4} \right) |E_\lambda| + \left( \frac{1}{A} + \mu \right) |E_\lambda|^2.
\end{equation}

\bigskip

\noindent
\underline{Step 3: from level set estimates to $L^p$ bounds.} 
Recall first that $\|F\|_{L^\infty(\mathbb{R} \times \mathbb{T}^d)} \lesssim N$ by the Sobolev embedding theorem. 
Choose next $\delta>0$. When estimating $|E_\lambda|$, three cases have to be distinguished:
\begin{itemize}
\item If $\lambda > N^{2/3+\delta} T^{1/12}$, then we choose $A = \frac{N^\delta}{ \lambda^{2}}$ (notice that $A < \frac{1}{N}$). The bound~\eqref{levelsetestimate} becomes then
$$
|E_\lambda| \preceq N^\delta \lambda^{-4}.
$$  
\item If $N^{2/3} < \lambda < N^{2/3+\delta} T^{1/12}$, then we choose $A = \frac{N^\delta}{\lambda^2}$, $\mu = \lambda^2 N^{-\delta/4}$, and appeal to~\eqref{levelsetestimate2} to obtain
$$
|E_\lambda| \lesssim N^\delta \left[ \lambda^{-4} + N^4 T \lambda^{-10} \right]{\lesssim N^{4+\delta} T\lambda^{-10}}.
$$
\item Finally, if $\lambda < N^{2/3}$, we rely on the Chebyshev inequality and the estimate $\|F\|_{L^{4}([0,T] \times \mathbb{T}^d)} \lesssim T^{1/4}$ (which follows from the $L^{4}$ bound of Bourgain-Demeter~\cite{BD2}) to obtain
$$
|E_\lambda| \lesssim T \lambda^{-4}.
$$
\end{itemize}
All in all, this gives for $4<p<10$
\begin{align*}
& \|F\|^p_{L^p([0,T] \times \mathbb{T}^d)} 
= p \int_{0}^{N} \lambda^{p-1} |E_\lambda|\,d\lambda \\
& \quad \preceq \int_0^{N^{2/3}} T \lambda^{p-5}\,d\lambda 
+ \int_{N^{2/3}}^{N^{2/3+\delta} T^{1/12}}N^\delta {N^4 T \lambda^{p-11}}\,d\lambda
+ \int_{N^{2/3+\delta} T^{1/12}}^{N} N^\delta \lambda^{p-5}\,d\lambda   \\
& \quad \preceq T(N^{2/3})^{p-4} + N^{p-4+\delta},
\end{align*}so Theorem ~\ref{thmweyl} is true.

\section{The counting argument through $\ell^2$ decoupling: {$p=14$}}

\label{section4}

We will prove Theorem~\ref{thmp8} for {$p \geq 14$}. {Note that the more general case $p\geq 10$ is proved in the next section, but the easier case here has a much simpler proof.} We want to estimate
$$
\| e^{it\widetilde{\Delta}} f \|_{L^p([0,T] \times \mathbb{T}^2)}^p
$$
for 
$$
f(y) = \sum \widehat{f}_k e^{2\pi i k \cdot y}.
$$ 
{By interpolation with $p=\infty$, we may assume $p=14$.} By splitting into dyadic scales, and absorbing the sum over scales in the subpolynomial factor, it suffices to consider the case where
$$
f(y) = \sum \widehat{f}_k {\mathbf{1}}_S(k) e^{2\pi i k \cdot y},
$$
where $S \subset B(0,N) \cap \mathbb{Z}^2$ and $\widehat{f}_k \sim 1$.

Our aim is to prove the following

\begin{thm}\label{14case}
With $f$ as above, and generically in $(\alpha_{ij})$,
$$
\| e^{it \widetilde{\Delta}} f \|_{L^p([0,T]\times \mathbb{T}^2)}^p
\preceq (N^4 + T) N^{p-8}|S|^{\frac{p}{2}}
\qquad \text{if }{p =14}.
$$
\end{thm}

\begin{proof}
Here and in the following, it will be convenient to denote
$$
\begin{pmatrix} \alpha_{11} & \alpha_{12} \\ \alpha_{12}&  \alpha_{22} \end{pmatrix} = \begin{pmatrix} 1 & \beta \\ \beta &  \alpha \end{pmatrix}.
$$
We define $\Omega^{p,N}_{a,b,A,B,C} $ to be the set of $(k_i,\ell_i) \in (\mathbb{Z}^2)^p$  such that $|k| < N$, $|\ell| < N$ and 
\begin{align}
\label{pinguin}
\sum_{i=1}^p (-1)^i k_i = a, \;\; \sum_{i=1}^p (-1)^i \ell_i = b, \;\; \sum_{i=1}^p (-1)^i k_i^2 = A, \;\;\sum_{i=1}^p (-1)^i \ell_i^2 = B,\;\;  \sum_{i=1}^p (-1)^i k_i \ell_i = C.
\end{align}
The result of Bourgain and Demeter~\cite{BD2} implies that
$$
|S^p \cap \Omega_{a,b,A,B,C}^{p,N}| \preceq N^{2p-10} |S| \qquad \mbox{for $p \geq 8$}.
$$
This implies that
\begin{equation}
\label{oriole}
|S^p \cap \Omega^{p, N}_{a,b,A,B,C}| \preceq N^{p-8} |S|^{\frac{p}{2}} \quad \mbox{{for $p=14$}}.
\end{equation}
Indeed, observe that one can first choose $k_i,\ell_i$ for $i = 1,\dots, \frac{p}{2} - 1$. There are $|S|^{\frac{p}{2}-1}$ possibilities. For the remaining $k_i,\ell_i$ (ie $i \geq \frac{p}{2}$, we use the estimate of Bourgain-Demeter, leading to a number of solutions $O(N^{p-8}|S|)$. Thus the total number of solutions is $O(N^{p-8}|S|^{\frac{p}{2}})$.

We can write
\begin{align*}
\| e^{it \widetilde{\Delta}} f \|_{L^p([0,T]\times \mathbb{T}^2)}^p
& = \sum_{A,B,C}  \sum_{(k_i,\ell_i) \in \Omega^{p,N}_{0,0,A,B,C}} \widehat{f}_{k_1, \ell_1} \dots \overline{\widehat{f}_{k_p, \ell_p}} {\mathbf{1}}_S(k_1, \ell_1) \dots  {\mathbf{1}}_S(k_p, \ell_p)
\frac{1- e^{2\pi iT(A+B\alpha + C \beta)}}{2\pi i(A+B\alpha + C \beta)} \\
& \lesssim \sum_{2^j > \frac{1}{T}} \sum_{|A + B \alpha + C \beta| \lesssim 2^{j}} 2^{-j} \sum_{\Omega^{p,N}_{0,0,A,B,C}} {\mathbf{1}}_S(k_1 \ell_1) \dots  {\mathbf{1}}_S(k_p \ell_p) \\
& = \sum_{2^j > \frac{1}{T}} \sum_{|A + B \alpha + C \beta| \lesssim 2^{j}} 2^{-j} |S^p \cap \Omega^{p,N}_{0,0,A,B,C}|.
\end{align*}
On the one hand, we can use~\eqref{oriole} to bound $|S^p \cap \Omega^{p,N}_{0,0,A,B,C}|$. On the other hand, by genericity of $\alpha$ and $\beta$,
$$
\sum_{|A + B \alpha + C \beta|<2^j} 1 \lesssim N^4 2^{j} + 1.
$$ {The above can be shown for example by summing in $(A,B,C)$ and integrating in $(\alpha,\beta)$ the characteristic function of $|A + B \alpha + C \beta|<2^j$, then applying Borel-Cantelli lemma.}
Overall, we find
$$
\| e^{it \widetilde{\Delta}} f \|_{L^p([0,T]\times \mathbb{T}^2)}^p \preceq \sum_{ \frac{1}{T}< 2^j \lesssim N^2} (N^4 + 2^{-j}) N^{p-8} |S|^{p/2} \preceq (N^4 + T) N^{p-8}|S|^{\frac{p}{2}}.
$$
\end{proof}

\section{The counting argument through Pall's bound: $p=8$}

\label{section5}

With the same ansatz for $f$ as in the previous section, we are going to prove the following theorem, which implies Theorem~\ref{thmp8} for $p=8$.

\begin{thm}\label{handsaw}
Generically in $(\alpha_{ij})$,
$$
\| e^{it \widetilde{\Delta}} f \|_{L^8([0,T]\times \mathbb{T}^2)}^8
\lesssim
|S^4|(N^4+NT).
$$
\end{thm}

Before starting, we first record an elementary estimate, which will be useful in the proofs below: for $a,\cdots,f\in\mathbb{R}$ and $\varepsilon>0$, we have
\begin{equation}\label{elem}
\begin{split}\mathrm{meas}\left\{(x,y)\in[1,2]^2:|ax+by+c|<\varepsilon\right\}&\lesssim \frac{\varepsilon}{\max(|a|,|b|,|c|)},\\
\mathrm{meas}\left\{(x,y)\in[1,2]^2:|ax+by+c|<\varepsilon,\,|dx+ey+f|<\varepsilon\right\}&\lesssim\frac{\varepsilon^2}{\max(|ae-bd|,|af-cd|,|bf-ce|)}.
\end{split}
\end{equation} We only prove the second inequality, since the first one is similar and easier. Let the given set be $A$, define the set
\[B=\left\{(x,y,z)\in[1,4]^3:|ax+by+cz|<2\varepsilon,\,|dx+ey+fz|<2\varepsilon\right\},
\] then by fixing one of the variables we easily deduce that \[\mathrm{meas}_{\mathbb{R}^3}(B)\lesssim\frac{\varepsilon^2}{\max(|ae-bd|,|af-cd|,|bf-ce|)}.\] Moreover we have
\[\mathrm{meas}_{\mathbb{R}^3}(B)\geq \int_{1}^2\mathrm{meas}_{\mathbb{R}^2}(A_z)\,\mathrm{d}z,\quad A_z:=\left\{(x,y)\in[1,4]^2:|a\frac{x}{z}+b\frac{y}{z}+c|<\frac{2\varepsilon}{z},\,|d\frac{x}{z}+e\frac{y}{z}+f|<\frac{2\varepsilon}{z}\right\},
\] and that $zA\subset A_z$ for $z\in[1,2]$, so a bound for $\mathrm{meas}_{\mathbb{R}^3}(B)$ implies the same bound for $\mathrm{meas}_{\mathbb{R}^2}(A)$.

\subsection{Pall's formula}

Our technical tool will be a formula from a paper of Pall~\cite{PAL}. Let $A',B',C'\in \Z$. We apply Pall's formula~(43) to  the quadratic form $\phi(x,y) = A'x^2+B'y^2+2C'xy$. This gives
\begin{lem}\label{redpoll}
	If  $A'B'-C^{\prime 2}> 0$, so that $\phi$ is positive definite, then
 \begin{equation*}
\sum_{\substack{ \sum_{i=1}^3 (k''_i)^2 = A',\\\sum_{i=1}^3 (\ell''_i)^2 = B',\\  \sum_{i=1}^3 k''_i \ell''_i = C' }} 1
=24 \cdot 2^\nu \prod_{p\mid 2k\Delta} \chi(p)
\end{equation*}
where $k_i'',\ell_i''\in\mathbb{Z}$, $k = \gcd(A',B',C')$, $ \Delta = (A'B'-C')/k^2$, we let $\nu$ be the number of distinct odd prime factors of $k\Delta$
and $\chi$ is defined in the comments after~(43) and at the start of page~358 in~\cite{PAL}.
\end{lem}

For the reader's convenience we repeat the definition of $\chi$ in full here. Given a prime $p$, let $p^{u_1} $ be the largest power of $p$ dividing $k$, let  $p^{u_2-u_1}$ be the largest power of $p$ dividing $\Delta$ and write $\delta_1 = \left\lfloor \frac{u_1+1}{2}\right \rfloor$. 
If $p=2$ then $\chi(p)=0$ or 1 according to the cases set out in Pall's formula~(44). If $p\neq 2$ we let $\phi_1 = \phi/k$ and adopt the convention that, if $p\mathrel{\mid}\Delta$, then $\legendre{\phi_1}{p}=\legendre{\phi_1(x,y)}{p}$ for any $x,y\in\Z$ such that $\phi_1(x,y)$ is prime to $p$. We further define quantities $\kappa_1$ and $\kappa_2$ by
\begin{equation*}
\begin{tabular}{ccl}
 $\kappa_1$		&	$\kappa_2$
 \\
 1			&	$\frac{1}{2}+\frac{1}{4}\left( 1 + \legendre{-kp^{-u_1}}{p}\legendre{\phi_1}{p} \right)(u_2-u_1)$		&if $u_1$ and $u_2$ even,
 \\
 $\frac{1}{2}\left(1+ \legendre{-kp^{-u_1}}{p}\legendre{\phi_1}{p} \right)$	&	$\frac{1}{4}\left( 1 + \legendre{-kp^{-u_1}}{p}\legendre{\phi_1}{p} \right)(u_2+1-u_1)$	&{if $u_1$ even and $u_2$ odd},
 \\
$\frac{1}{2}\left(1+ \legendre{-k\Delta p^{-u_2}}{p}\legendre{\phi_1}{p} \right)$	&	0	&{if $u_1$ odd and $u_2$ even},
 \\
 $\frac{1}{2}\left(1+ \legendre{-\Delta p^{u_1-u_2}}{p}\right)$	&	0	&{if $u_1$ and $u_2$ odd}.
 \end{tabular}
\end{equation*}
Then we can set
\begin{equation}
\label{linnet}
\chi(p) = \kappa_1 (p^{\delta_1}-1)/(p-1) + \kappa_2 p^{\delta_1}.
\end{equation}

\begin{cor}\label{twite}
If $A'B'-C^{\prime 2}> 0$ and $A',B',C'\lesssim N^2$ then
$$
\sum_{\substack{ \sum_{i=1}^3 (k''_i)^2 = A',\\\sum_{i=1}^3 (\ell''_i)^2 = B',\\  \sum_{i=1}^3 k''_i \ell''_i = C' }} 1 \preceq h,
$$
where $h$ is the largest natural number such that
$
h^2 \mid \gcd(A',B',C')
$. 
\end{cor}

\begin{proof}
It is immediate from \eqref{linnet} and the preceding table that for odd $p$ we have
 $$
 \chi(p) \leq
 \begin{cases}
  2(u_2+1-u_1)p^{\delta_1} & u_1 \text{ even},
  \\
  2p^{\delta_1-1} & u_1 \text{ odd}.
 \end{cases}
 $$
Hence, where as in Lemma~\ref{redpoll} we let $\nu$ be the number of distinct odd prime factors of $k\Delta$,
\begin{align*}
\prod_{ p \mathrel{\mid} k\Delta } \chi(p)& \leq 2^\nu \left( \prod_{ p \mathrel{\mid} k\Delta } (u_2+1-u_1)\right)\left( \prod_{ p \mathrel{\mid} k\Delta } p^{\left\lfloor \frac{u_1}{2}\right\rfloor} \right)
\\
& = 2^\nu d(\Delta) h
\end{align*}
on recalling the definitions of $k$ and $u_1$ in Lemma~\ref{redpoll} and the subsequent comments. The result follows by the divisor bound.
\end{proof}

\subsection{Preliminaries}
We start with a discussion valid for any even $p$. Recall that
$$
\Omega^{p,N}_{a,b,A,B,C} = \{ (k_i,\ell_i)_{i=1,\dots,p} \in (\mathbb{Z}^2)^p \; \mbox{such that} \; |k_i| < N, \; |\ell_i| <N, \; \sum_{i=1}^p (-1)^i k_i = a, \text{etc.} \}.
$$
Define a version without the alternating signs,
$$
\Omega^{q,N,+}_{a,b,A,B,C} = \{ (k_i,\ell_i)_{i=1,\dots,q} \in (\mathbb{Z}^2)^q \; \mbox{such that} \; |k_i| < N, \; |\ell_i| <N, \; \sum_{i=1}^q k_i = a, \text{etc.} \}.
$$
Recall that
$$
\| e^{it \widetilde{\Delta}} f \|_{L^p([0,T]\times \mathbb{T}^2)}^p \lesssim \sum_{2^j > \frac{1}{T}} \sum_{|A + B \alpha + C \beta| \lesssim 2^{j}} 2^{-j} |S^p \cap \Omega^{p,N}_{0,0,A,B,C}|
$$
and observe (for $p/2$ even)
\begin{align}
\MoveEqLeft[2]\| e^{it \widetilde{\Delta}} f \|_{L^p([0,T]\times \mathbb{T}^2)}^{p}
\nonumber
\\
&\lesssim
\sum_{2^j > \frac{1}{T}} \sum_{\substack{a,b\lesssim N \\ A_i,B_i,C_i \lesssim N^2 \\ |(A_1-A_2) + (B_1-B_2) \alpha + (C_1-C_2) \beta| \lesssim 2^{j}}} 2^{-j} |S^{p/2} \cap \Omega^{p/2,N,+}_{a,b,A_1,B_1,C_1}| \cdot |S^{p/2} \cap \Omega^{p/2,N,+}_{a,b,A_2,B_2,C_2}|
\nonumber
\\
&\lesssim
\sum_{\frac{1}{T}< 2^j \lesssim N^2} 2^{-j} \lvert S^{p/2} \rvert \sup_{\substack{a,b\lesssim N \\ \tau\lesssim N^2} } \sum_{\substack{A_1,B_1,C_1 \lesssim N^2 \\ |A_1 + B_1 \alpha + C_1 \beta-\tau| < 2^{j}}}  \lvert\Omega^{p/2,N,+}_{a,b,A_1,B_1,C_1}\rvert.
\label{jacana}
\end{align} 

Restricting the discussion to $p=8$ from now on, we may write $k'_i = 4k_i-a$ and $\ell'_i=4\ell_i-b$ to obtain
\begin{equation}
\begin{gathered}
\sum_{i=1}^4 k_i = a, \;\; \sum_{i=1}^4 \ell_i = b, \;\; \sum_{i=1}^4 k_i^2 = A_1, \;\;\sum_{i=1}^4 \ell_i^2 = B_1,\;\;  \sum_{i=1}^4 k_i \ell_i = C_1
\\
\iff
\\
\sum_{i=1}^4 k'_i = 0, \;\; \sum_{i=1}^4 \ell'_i = 0, \;\; \sum_{i=1}^4 (k'_i)^2 = 16A_1-4a^2, \\
\sum_{i=1}^4 (\ell'_i)^2 = 16B_1-4b^2,\;\;  \sum_{i=1}^4 k'_i \ell'_i = 16C_1-4ab.
\end{gathered}
\label{elf_owl}
\end{equation}
So, writing
\begin{equation}
\begin{gathered}
k_1'' = k_2'+k_3',\;k_2''=k'_1+k'_3,\;k''_3=k_1'+k_2',\\
\ell_1'' = \ell_2'+\ell_3',\;\ell_2''=\ell'_1+\ell'_3,\;\ell''_3=\ell_1'+\ell_2',\\
A' = 16A_1-4a^2,\;B'=16B_1-4b^2,\; C'=16C_1-4ab,
\end{gathered}
\label{kookaburra}
\end{equation}
we get an injection
\begin{equation}
\Omega^{4,N,+}_{a,b,A_1,B_1,C_1}
\hookrightarrow
\{ (k''_i,\ell''_i)_{i=1,2,3} :
\sum_{i=1}^3 (k''_i)^2 = A', \;\;\sum_{i=1}^3 (\ell''_i)^2 = B',\;\;  \sum_{i=1}^3 k''_i \ell''_i = C'.\}
\label{sing_kookaburra}
\end{equation}
In particular, we may conclude from \eqref{jacana} that
\begin{align}
\| e^{it \widetilde{\Delta}} f \|_{L^8([0,T]\times \mathbb{T}^2)}^8
&\lesssim
\sum_{\frac{1}{T}< 2^j \lesssim N^2}
2^{-j} \lvert S^4\rvert \sup_{\tau'\lesssim N^2 }
\sum_{\substack{A',B',C' \lesssim N^2 \\ |A' + B' \alpha + C' \beta-\tau'| < 2^{j}  }}
\sum_{\substack{ \sum_{i=1}^3 (k''_i)^2 = A',\\\sum_{i=1}^3 (\ell''_i)^2 = B',\\  \sum_{i=1}^3 k''_i \ell''_i = C' }} 1
\nonumber
\\
& \preceq
\sum_{\frac{1}{T}< 2^j \lesssim N^2}
2^{-j}
\lvert S^4\rvert \sup_{\tau'\lesssim N^2 }
\Bigg(
(1+2^j)N +
\sum_{\substack{A',B',C' \lesssim N^2 \\ |A' + B' \alpha + C' \beta-\tau'| < 2^{j} \\ A'B'\neq 0 }}
\sum_{\substack{ \sum_{i=1}^3 (k''_i)^2 = A',\\\sum_{i=1}^3 (\ell''_i)^2 = B',\\  \sum_{i=1}^3 k''_i \ell''_i = C' }} 1
\Bigg)
\label{sandpiper}
\end{align}
 by  separating out the terms where $k'' = 0$ or $\ell''=0$ .

\subsection{The case $A'B'=(C')^2$}\label{sec:8_variables_singular_case}

The innermost sum in \eqref{sandpiper} is largest for values of $A', B', C'$ for which $A'B'=(C')^2$. We will need to deal with these degenerate terms separately. If $\sum_{i=1}^3 (k''_i)^2 \sum_{i=1}^3 (\ell''_i)^2 =(\sum_{i=1}^3 k''_i \ell''_i)^2 \neq 0$ then we must have
$$
k''_i = pz_i,\;\;\ell''_i=qz_i
$$
for some unique $z_i, p,q$ with  $p,q>0$ and $\gcd(p,q)=1$. Thus
\begin{align}
\sum_{\substack{A',B',C' \lesssim N^2 \\ |A' + B' \alpha + C' \beta-\tau'| < 2^{j} \\ A' B'\neq 0 \\ A'B'=(C')^2 }}
\sum_{\substack{ \sum_{i=1}^3 (k''_i)^2 = A',\\\sum_{i=1}^3 (\ell''_i)^2 = B',\\  \sum_{i=1}^3 k''_i \ell''_i = C' }} 1
&\leq
\sum_{\substack{p,q,m\in\N \\ \gcd(p,q) = 1 \\ p, q \lesssim N/\sqrt{m}
	 \\ |(p^2+ q^2 \alpha + pq \beta)m-\tau'| < 2^{j} }} \sum_{\substack{ z_i \\ z_1^2+z_2^2+z_3^2 = m}} 1
 \nonumber
 \\
 &\preceq
 \sum_{\substack{p,q,m\in\N\\ \gcd(p,q) = 1\\ mp^2, mq^2 \lesssim N^2
 		\\ |(p^2+ q^2 \alpha + pq \beta)m-\tau'| < 2^{j} }} \sqrt{m}.
 	\label{goshawk}
\end{align}
To finish this case, it is enough to involve the genericity of $\alpha,\beta$ as follows:
\begin{lem}\label{8-variables-singular-case}
	For almost all $\alpha,\beta$ and all $m,m',p,p',q,q'$ satisying $(mp^2,mq^2)\neq (m'(p')^2,m'(q')^2)$ and $mp^2,mq^2,m'(p')^2,m'(q')^2\leq N^2$ we have
	$$
	\lvert 
	m(p^2+q^2\alpha + pq\beta)
	-m'((p')^2+(q')^2\alpha + p'q'\beta )
	\rvert
	\succeq_{\alpha,\beta} N^{-2}.
	$$
\end{lem}
\begin{proof} Fix $\tau>0$. Denoting $E_{m,p,q,m',p',q'} = \{ (\alpha,\beta) \; \mbox{such that} \; |m(p^2 + \alpha q^2 + pq \beta) - m'((p')^2 + \alpha (q')^2 + \beta p' q' )| < N^{-2-\tau} \}$, we get, by applying (\ref{elem}), that
$$
|E_{m,p,q,m',p',q'}| \leq N^{-2-\tau} \min \left( \frac{1}{mq^2-m'(q')^2} , \frac{1}{mp^2 - m' (p')^2} \right).
$$ Now we will calculate
\[\sum_{m,p,q,m',p',q'} |E_{m,p,q,m',p',q'}|\sim N^{-2-\tau}\sum_{1\leq d\lesssim N^2}\frac{1}{d}\sum_{\substack{(m,p,q,m',p',q'):\\|mq^2-m'(q')^2|=d\\|mp^2 - m' (p')^2|\leq d}}1.
\] For fixed $C$, when $(m,m',q,q')$ is also fixed, we have $1\leq p\lesssim N/\sqrt{m}$, and for each $p$, we have
\[\frac{\sqrt{\max(mp^2-d,0)}}{\sqrt{m'}}\leq p'\leq\frac{\sqrt{mp^2+d}}{\sqrt{m'}},
\] so the number of choices for $p'$ is at most
\[1+\frac{\sqrt{mp^2+d}-\sqrt{\max(mp^2-d,0)}}{\sqrt{m'}}\sim 1+\frac{d}{\sqrt{m'(mp^2+d)}}.
\] Therefore, for fixed $(m,m',q,q')$ we have
\[\sum_{(p,p'):|mp^2 - m' (p')^2|\leq d}1\lesssim\sum_{p\lesssim N/\sqrt{m}}\bigg(1+\frac{d}{\sqrt{m'(mp^2+d)}}\bigg)\preceq\frac{N}{\sqrt{m}}+\frac{d}{\sqrt{mm'}}.
\] Now by definition of $d$, we have 
\[\sum_{1\leq d\lesssim N^2}\frac{1}{d}\sum_{\substack{(m,m',q,q'):\\|mq^2-m'(q')^2|=d}}\frac{d}{\sqrt{mm'}}\lesssim \sum_{\substack{1\leq m,m'\lesssim N^2\\1\leq q\lesssim N/\sqrt{m},1\leq q'\lesssim N/\sqrt{m'}}}\frac{1}{\sqrt{mm'}}\lesssim \sum_{1\leq m,m'\lesssim N^2}\frac{1}{\sqrt{mm'}}\frac{N^2}{\sqrt{mm'}}\preceq N^2.
\] Moreover, when $m$ is fixed there are at most $N/\sqrt{m}$ choices for $q$, and when $(m,q)$ is fixed, there are most $O(N^{\varepsilon})$ choices for $(m',q')$ such that $|mq^2-m'(q')^2|=d$ due to divisor estimates, so
\[\sum_{1\leq d\lesssim N^2}\frac{1}{d}\sum_{\substack{(m,m',q,q'):\\|mq^2-m'(q')^2|=d}}\frac{N}{\sqrt{m}}\preceq \sum_{1\leq d\lesssim N^2}\frac{1}{d}\sum_{1\leq m\lesssim N^2}\frac{N}{\sqrt{m}}\frac{N}{\sqrt{m}}\preceq N^2.
\] This implies that
\[\sum_{m,p,q,m',p',q'} |E_{m,p,q,m',p',q'}|<\infty,
\]so the lemma of Borel-Cantelli then gives the result. {Namely, almost all $(\alpha,\beta)$ belongs to only finitely many $E_{m,p,q,m',p',q'}$, so the desired inequality holds with some constant depending on $(\alpha,\beta)$ by treating the finitely many sets individually.
} 
\end{proof}

By \eqref{goshawk} and the Lemma we get
\begin{equation}
\label{flamingo}
\sum_{\substack{A',B',C' \lesssim N^2 \\ |A' + B' \alpha + C' \beta-\tau'| < 2^{j} \\ A' B'\neq 0 \\ A'B'=(C')^2 }}
\sum_{\substack{ \sum_{i=1}^3 (k''_i)^2 = A',\\\sum_{i=1}^3 (\ell''_i)^2 = B',\\  \sum_{i=1}^3 k''_i \ell''_i = C' }} 1
\preceq
(2^jN^2+1)N
\end{equation}
which is enough for us.

\subsection{The case $A'B'\neq(C')^2$, main argument}

Recall the bound from Corollary~\ref{twite} above. Setting $h=h(A',B',C')$, $U = h^{-2} A'$, $V=h^{-2} B'$, and $W = h^{-2} C'$ we deduce from it that
\begin{equation}
\sum_{\substack{A',B',C' \lesssim N^2 \\ |A' + B' \alpha + C' \beta-\tau'| < 2^{j} \\ A'B'\neq(C')^2 }}
\sum_{\substack{ \sum_{i=1}^3 (k''_i)^2 = A',\\\sum_{i=1}^3 (\ell''_i)^2 = B',\\  \sum_{i=1}^3 k''_i \ell''_i = C' }} 1
\preceq
\sum_{h\lesssim N} 
\sum_{\substack{U,V,W \lesssim N^2/h^2 \\ |U + V \alpha + W \beta-\tau'/h^2| < 2^{j}/h^2  \\  UV-W^2 > 0 }} h.
\end{equation}
Observe that since $|U + V \alpha + W \beta|\succeq_{\alpha,\beta} h^4/N^4$ for all $U,V,W \lesssim N^2/h^2 $ not all zero (by {integrating in $\alpha$ and $\beta$ exploiting genericity, in the same way as in the proof of Theorem \ref{14case})}, the sum over those $h\leq N^{1/2}$ satisfies
\begin{equation*}
\sum_{h\leq N^{1/2}}
\sum_{\substack{U,V,W \lesssim N^2/h^2 \\ |U + V \alpha + W \beta-\tau'/h^2| < 2^{j}/h^2 \\ UV-W^2 > 0  }} h
\preceq
\sum_{h\lesssim N^{1/2}} (2^jN^4/h^6+1)h
\lesssim 2^j N^4+N.
\end{equation*}
This is satisfactory, and it remains to treat $ N^{1/2}\leq h \lesssim N$. Our proof does not use the full strength  of the condition $UV-W^2 > 0$ but only the weaker $(U,V,W)\neq (0,0,0)$; this should heuristically make no difference in the size of the sum and we do not know any way to take advantage of the full condition $UV-W^2 > 0$. The result which is required is then the following.
\begin{lem}\label{gazunk}
Let
 $$
 \Phi(\alpha,\beta, \tau') =
 \sum_{h\sim K} 
\sum_{\substack{U,V,W \lesssim N^2/h^2 \\ |U + V \alpha + W \beta-\tau'/h^2| < 2^{j}/h^2 \\  (U,V,W)\neq (0,0,0) }} 1
 $$
 For almost all $\alpha,\beta$ and any $\tau'\in\R$, $j\in \Z$, $K\in [N^{1/2},N]$ we have 
$$
\Phi(\alpha,\beta, \tau') \preceq_{\alpha,\beta} 2^j N^4/K+N/K.
$$
\end{lem}
Combining this lemma with the previous two bounds will show that for generic $\alpha,\beta$ we have
\begin{equation}
	\label{buzzard}
	\sum_{\substack{A',B',C' \lesssim N^2 \\ |A' + B' \alpha + C' \beta-\tau'| < 2^{j} \\ A'B'\neq(C')^2 }}
	\sum_{\substack{ \sum_{i=1}^3 (k''_i)^2 = A',\\\sum_{i=1}^3 (\ell''_i)^2 = B',\\  \sum_{i=1}^3 k''_i \ell''_i = C' }} 1
	\preceq_{\alpha,\beta}
	2^j N^4+N.
\end{equation}
To motivate the proof of the lemma, we sketch some attacks on the problem based on the proof of \eqref{flamingo}. Imitating Lemma~\ref{8-variables-singular-case}, one might try to show that for generic $(\alpha, \beta)$, all $h,h'\sim K$, $U, U',V,V',W,W' \lesssim N^2/h^2 $, and some appropriate $\rho$ we have
\begin{equation}\label{boobrie}
\left\lvert h^2(U + V \alpha + W \beta)-h^{\prime 2} (U'+V'\alpha+W'\beta)\right\rvert
\gtrsim_{\alpha,\beta} \rho
\end{equation}
whenever the left-hand side is nonzero. It would follow that the number of values of $h^2(U + V \alpha + W \beta)$ such that $\left\lvert h^2(U + V \alpha + W \beta)-\tau'\right\rvert< 2^j$ is at most $\preceq_{\alpha,\beta}2^j/\rho$. Taking an optimistic view, we might get a bound $\preceq_{\alpha,\beta}2^j K/\rho$ for the sum in Lemma~\ref{gazunk}.

What value of $\rho$ might be possible? The left-hand side of \eqref{boobrie} is of the form $G+H\alpha+J\beta$ for integers $G,H,J\lesssim N^{2}$. It is elementary that if $K\sim N^{1/2}$, any triple of integers $G,H,J\lesssim_\epsilon N^{2-\epsilon}$ will occur for some $h,h',U,U'$, etc. So if $K\sim N^{1/2}$ then there are $h,h',U,U',$ etc.\@ with
$$
\left\lvert h^2(U + V \alpha + W \beta)-h^{\prime 2} (U'+V'\alpha+W'\beta)\right\rvert \lesssim_\epsilon N^{-4+\epsilon}.
$$
Thus $\rho\lesssim_\epsilon N^{-4+\epsilon}$ in this case. The optimistic bound $2^j K/\rho$ for the sum in Lemma~\ref{gazunk} is then of size at least $\gtrsim_\epsilon 2^j K N^{4-\epsilon}$. But this is bigger than the bound in (\ref{buzzard}) by a factor of $KN^{-\epsilon}$.

We can try to refine the argument: rather than prove $\left\lvert h^2(U + V \alpha + W \beta)-h^{\prime 2} (U'+V'\alpha+W'\beta)\right\rvert < \rho$ has no solutions, we can seek a bound for the number of solutions. More precisely, let
\begin{equation*}\label{eblis}
\Sigma(\alpha,\beta,\tau')
=
\sum_{h,h'\sim K} 
\sum_{\substack{|U + V \alpha + W \beta-\tau'/h^2| < 2^{j}/h^2  \\ |U' + V' \alpha + W' \beta-\tau'/h^{\prime 2}| < 2^{j}/h^{\prime 2} \\
		 h^2(U,V,W)\neq h^{\prime 2}(U',V', W')
}} 1
\end{equation*}
and write
\begin{equation}
	\bigg( \sum_{h\sim K} 
	\sum_{\substack{U,V,W \lesssim N^2/h^2 \\ |U + V \alpha + W \beta-\tau'/h^2| < 2^{j}/h^2  }} h\bigg)^2
	\lesssim K^2\Sigma(\alpha,\beta,\tau')+
	\sum_{h,h'\sim K} 
	\sum_{\substack{|U + V \alpha + W \beta-\tau'/h^2| < 2^{j}/h^2  \\ |U' + V' \alpha + W' \beta-\tau'/h^{\prime 2}| < 2^{j}/h^{\prime 2} 
			\\
			 h^2(U,V,W)= h^{\prime 2}(U',V', W')
		}} hh'.
	\label{whipoorwill}
\end{equation}
The second term captures the diagonal contribution from $h^2(U,V,W) = h^{\prime 2}(U',V', W')$, which would otherwise cause problems later; we might hope that 
 this term is negligible. We would then need to bound $\Sigma(\alpha,\beta,\tau')$ for almost all $\alpha,\beta$, uniformly in $\tau'$. We can eliminate $\tau'$ using the same idea as Lemma~\ref{8-variables-singular-case}: we observe that $\Sigma(\alpha,\beta,\tau')\leq\Sigma'(\alpha,\beta) $, where we set
$$
\Sigma'(\alpha,\beta)
=
\sum_{h,h'\sim K} 
\sum_{\substack{\left\lvert h^2(U + V \alpha + W \beta)-h^{\prime 2} (U'+V'\alpha+W'\beta)\right\rvert < 2^j \\
		  h^2(U,V,W)\neq h^{\prime 2}(U',V', W')
 }} 1.
$$
Using the Borel-Cantelli lemma we find that for generic $\alpha,\beta$ we have
\begin{multline*}
\Sigma'(\alpha,\beta)
\preceq_{\alpha,\beta}
\int
\Sigma'(\alpha',\beta') 
\, d\alpha' d\beta' =
\sum_{h,h'\sim K} \sum_{
h^2(U,V,W)\neq h^{\prime 2}(U',V', W')  }
\\
\operatorname{measure}\big\{ (\alpha', \beta') : \left\lvert h^2(U + V \alpha' + W \beta')-h^{\prime 2} (U'+V'\alpha'+W'\beta')\right\rvert < 2^j
\big\}.
\end{multline*}
The condition $
h^2(U,V,W)\neq h^{\prime 2}(U',V', W')$ removes the terms for which the measure above is largest, which could otherwise have dominated the sum. But even for typical $h,h',U,U',\dotsc$ this measure will have size $\gtrsim 2^j/N^2$. Consquently the best upper bound for $\Sigma'(\alpha,\beta)$ which we could hope to prove is
$$
\Sigma'(\alpha,\beta)
\preceq_{\alpha,\beta}
2^j N^{12}/K^{10}.
$$
Via \eqref{whipoorwill} this would lead to a bound of at best $2^{j/2} N^6/K^4$ for the sum in Lemma~\ref{gazunk}, which is not sufficient. For example when $K=N^{1/2}$ and $2^j=N^{-3}$ that bound is at least $N^{5/2}$, while (\ref{buzzard}) requires a bound of $\lesssim_{\alpha, \epsilon}N^{1+\epsilon}$.

The proof we give of Lemma~\ref{gazunk} is nonetheless very close to the sketch above. Instead of the square in \eqref{whipoorwill} we will take the cube of the sum we hope to bound. We will split off a kind of ``diagonal contribution", and the remaining term will be estimated by eliminating $\tau'$ and applying the Borel-Cantelli lemma.

\begin{proof}[Proof of Lemma~\ref{gazunk}]
We may assume $2^j\geq N^{-3}$, since otherwise the term $N/K$ on the right hand side dominates so the case can be treated as if $2^j\sim N^{-3}$. We take the cube:
	\begin{equation*}
	 \Phi(\alpha,\beta, \tau')^3
	=
	\sum_{h_0,h_1,h_2\sim K} 
	\sum_{\substack{U_i,V_i,W_i \lesssim N^2/K^2\;\;(i=0,1,2) \\ |U_i + V_i\alpha + W_i \beta-\tau'/h_i^2| < 2^{j}/h_i^2  \;\;(i=0,1,2)\\ (U_i,V_i,W_i)\neq (0,0,0)   \;\;(i=0,1,2)}} 1.
	\end{equation*}
	We split the terms according to whether the $h_i^2(U_i,V_i,W_i)$ lie on a line through the origin, or else on a line which is not through the origin, or neither. Write $\times$ for the vector product of two column vectors. Observe that the $h_i^2(U_i,V_i,W_i)$ are collinear iff the quantity
	\begin{equation}
		\label{billdad}
		\Delta = \left(h_2^2(U_2,V_2,W_2)-h_0^2(U_0,V_0,W_0)\right)^T\times \left(h_1^2(U_1,V_1,W_1)-h_0^2(U_0,V_0,W_0)\right)^T
	\end{equation}
	vanishes.
	So the promised splitting of the sum is
	\begin{align}
	\label{marshharrier}
 \Phi(\alpha,\beta, \tau')^3\lesssim{}&
	\sum_{h_0,h_1,h_2\sim K} 
	\sum_{\substack{U_i,V_i,W_i \lesssim N^2/K^2\;\;(i=0,1,2) \\ \left\lvert h_i^2(U_i + V_i\alpha + W_i \beta)-\tau'\right\rvert < 2^{j}  \;\;(i=0,1,2) \\
			h_i^2(U_i,V_i,W_i) =  \lambda_i(U_0,V_0, W_0) \;\; (\lambda_i \in \Q) \\ (U_0,V_0,W_0)\neq (0,0,0)  }}
	1
	\\
	&+
	\sum_{h_0,h_1,h_2\sim K} 
	\sum_{\substack{U_i,V_i,W_i \lesssim N^2/K^2\;\;(i=0,1,2) \\ \left\lvert h_i^2(U_i + V_i\alpha + W_i \beta)-\tau'\right\rvert < 2^{j}  \;\;(i=0,1,2) \\
			(U_0,V_0,W_0)^T\times  (U_1,V_1, W_1)^T \neq 0 \\ \Delta = 0 }}
	\nonumber
 1
 	\\
&+
 \sum_{h_0,h_1,h_2\sim K} 
 \sum_{\substack{U_i,V_i,W_i \lesssim N^2/K^2\;\;(i=0,1,2)  \\ \left\lvert h_i^2(U_i + V_i\alpha + W_i \beta)-\tau'\right\rvert < 2^{j}  \;\;(i=0,1,2) \\
 		 \Delta \neq 0 }}
 1.
 \nonumber
	\end{align}

	The first term on the right-hand side includes the diagonal contribution, when the $h_i^2(U_i,V_i,W_i)$ are all equal. Letting $(\hat{U}_0,\hat{V}_0,\hat{W}_0)=(U_0,V_0,W_0)/\gcd(U_0,V_0,W_0)$, this first term is bounded by
	\begin{multline*}
	\sum_{h_0 \sim K} 
	\sum_{\substack{U_0,V_0,W_0 \lesssim N^2/K^2 \\ \left\lvert h_0^2(U_0 + V_0\alpha + W_0 \beta)-\tau'\right\rvert < 2^{j}  \\ (U_0,V_0,W_0)\neq (0,0,0) 
	 }}
 \left(
	\sum_{\substack{ \lvert y(\hat{U}_0+\hat{V}_0\alpha +\hat{W}_0\beta)-\tau'| < 2^{j}  \;\;(y\in \Z) }}
	1
	\right)^2
	\\
	\leq
	\sum_{h_0 \sim K} 
	\sum_{\substack{U_0,V_0,W_0 \lesssim N^2/K^2 \\ \left\lvert h_0^2(U_0 + V_0\alpha + W_0 \beta)-\tau'\right\rvert < 2^{j}  \\ (U_0,V_0,W_0)\neq (0,0,0) 
	}}
	\left(
	\frac{2^{j+1}}{ \lvert \hat{U}_0+\hat{V}_0\alpha +\hat{W}_0\beta\rvert }	
	\right)^2 + 1.
	\end{multline*}
	For generic $\alpha,\beta$ we have $\lvert \hat{U}_0+\hat{V}_0\alpha +\hat{W}_0\beta\rvert\succeq_{\alpha,\beta}K^4/N^4$ and so the sum on the last line is
	\begin{equation}\label{turul}
	\preceq_{\alpha,\beta}
	\left( \frac{ 2^{2j} N^8}{K^8} + 1 \right)
	\Phi(\alpha,\beta, \tau').
	\end{equation}
	This will suffice for the first term. A variation of the same argument produces the following lemma, which is enough to treat the second term in \eqref{marshharrier}.
\begin{lem}\label{zapdos}
	For generic $\alpha,\beta$, $2^j\geq N^{-3}$ and $\tau'\in\R$, $K\in[N^{1/2},N]$ and $m,c\in \Z^3$ with $\lvert m \rvert, \lvert c\rvert \lesssim N^2$, $\gcd(m_1,m_2,m_3)=1$ and  $m^T\times c^T\neq 0$, we have
	$$
	\sum_{h_2\sim K} 
	\sum_{\substack{U_2,V_2,W_2 \lesssim N^2/K^2  \\ \left\lvert h_2^{ 2}(U_2 + V_2 \alpha + W_2 \beta)-\tau'\right\rvert < 2^j\\ 
			h_2^{ 2}(U_2 ,V_2 , W_2 ) = mx+c \;\;(x\in\Z) }}
	1
	\preceq_{\alpha,\beta}
	 \frac{2^j  N^4}{K^2} + 1.
	$$
\end{lem}
We also need a lemma to treat the third term in \eqref{marshharrier}; this will be proved in the next section using a Borel-Cantelli argument.
	\begin{lem}\label{cinnamologus}
		Define		
		\begin{equation*}
		\Sigma_0(\alpha,\beta)
		=	\sum_{h_0,h_1,h_2\sim K} 
		\sum_{\substack{U_i,V_i,W_i \lesssim N^2/K^2\;\;(i=0,1,2)  \\ \left\lvert h_i^2(U_i + V_i \alpha + W_i \beta)-h_0^2 (U_0+V_0\alpha+W_0\beta)\right\rvert < 2^j \;\;(i=1,2)\\  \Delta \neq 0 }}
		1.
		\end{equation*}
	For generic $\alpha,\beta$, $2^j\geq N^{-3}$ and $K\in[N^{1/2},N]$ we have $$
	\Sigma_0(\alpha,\beta)
		\preceq_{\alpha,\beta}
		1+\frac{2^{3j}N^{12}}{K^{3}}
		.
		$$
	\end{lem}
	Inserting \eqref{turul} and the two lemmas into \eqref{marshharrier} gives
	\begin{equation*}
 \Phi(\alpha,\beta, \tau')^3
	\preceq
	\left( \frac{ 2^{2j} N^8}{K^8} + 1 \right)  \Phi(\alpha,\beta, \tau')+
	\left( \frac{ 2^{j} N^4}{K^2} + 1 \right)
	\Phi(\alpha,\beta, \tau')^2+1
	+\frac{2^{3j}N^{12}}{K^{3}},
	\end{equation*}
	from which we obtain
	$$
	\Phi(\alpha,\beta, \tau')
	\preceq_{\alpha,\beta}
	\frac{ 2^{j} N^4}{K^4}+
	\frac{ 2^{j} N^4}{K^2}+1
	+\frac{2^{j}N^{4}}{K}.
	$$
	This proves the lemma.
\end{proof}

\subsection{The case $A'B'\neq(C')^2$, auxiliary lemmas}
	
	\begin{proof}[Proof of Lemma~\ref{zapdos}]
		Suppose $mx+c \equiv 0 \pmod{ h_2^{ 2}}$. From this we draw two conclusions. First,
		$$
		\gcd(x,h_2^{ 2}) \mid \gcd(c_1,c_2,c_3).
		$$
		Second,
		$mx\times\left(\frac{c}{\gcd(x,h_2^{ 2})}\right) \equiv 0 \mod{h_2^{ 2}}$ 
	and so
	$$
	\frac{m\times c}{\gcd(x,h_2^{ 2})} \equiv 0 \mod{\frac{h_2^{ 2}}{\gcd(x,h_2^{ 2})}}.
	$$
	From these and the fact that $\lvert m \rvert, \lvert c\rvert \lesssim N^2$ and  $m^T\times c^T\neq 0$, we find that $\gcd(x,h_2^{ 2})$ and $h_2^{ 2}$ are  determined up to $\preceq 1$ possibilities by $m$ and $c$. Thus the sum in the lemma is 
	$$
	\preceq
	\max_{h_2\sim K} 
	\sum_{\substack{U_2,V_2,W_2 \lesssim N^2/K^2\;\;(i=0,1,2)  \\ \left\lvert h_2^{ 2}(U_2 + V_2 \alpha + W_2 \beta)-\tau'\right\rvert < 2^j\\ 
			h_2^{ 2}(U_2 ,V_2 , W_2 ) = mx+c \;\;(x\in\Z) }}
	1.
	$$
Let $x_0\in\mathbb{Z}$ be fixed such that $mx_0+c\equiv 0\pmod {h_2^2}$ (such $x_0$ exists, otherwise the sum will be zero). Let $x'=x-x_0$ and $c'=c+mx_0$, then we have $mx'\equiv 0\pmod{h_2^2}$. As $\gcd(m_1,m_2,m_3)=1$, we conclude that $x' \equiv 0\pmod{h_2^{ 2}}$. Then the expression above is bounded by
	$$
	\max_{h_2\sim K} 
	\sum_{\substack{
			x' \equiv 0\pmod{h_2^{ 2}}
			\\ 
			\left\lvert x'(m_1+m_2\alpha+m_3\beta) + (c'_1+c'_2\alpha+c'_3\beta) -\tau'\right\rvert < 2^{j}  }}
	1
	\lesssim
	\max_{h_2\sim K} \left( \frac{2^j}{\lvert h_2^{ 2} (m_1+m_2\alpha+m_3\beta) \rvert}+1 \right).
	$$
	Now for generic $\alpha, \beta$ this is
		\begin{equation*}
		\preceq_{\alpha,\beta} \frac{2^j \lvert m \rvert^{2+\epsilon}}{K^2}+1
		\end{equation*}
from which the claim follows.	
	\end{proof}
	
	\begin{proof}[Proof of Lemma~\ref{cinnamologus}]
		By letting $K$ and $N$ range over powers of 2, the Borel-Cantelli lemma implies that for almost all $\alpha,\beta$ we have
\begin{align*}
\Sigma_0(\alpha,\beta)
\preceq_{\alpha,\beta}{}&1+
\int
\Sigma_0(\alpha',\beta') 
\, d\alpha' d\beta'
\\
{}={}&1+
\sum_{h_0,h_1,h_2\sim K} 
\sum_{\substack{U_i,V_i,W_i \lesssim N^2/K^2\;\;(i=0,1,2)  \\  \Delta \neq 0 }}
\\
&\operatorname{measure}\big\{ (\alpha', \beta') :
\left\lvert h_i^2(U_i + V_i \alpha' + W_i \beta')-h_0^2 (U_0+V_0\alpha'+W_0\beta')\right\rvert < 2^j\;\; (i=1,2)
\big\}.
\end{align*}
In light of (\ref{elem}) and \eqref{billdad}, this implies
\begin{equation}\label{lyrebird}
\Sigma_0(\alpha,\beta)
\preceq_{\alpha,\beta} 1+
\sum_{h_0,h_1,h_2\sim K} 
\sum_{\substack{U_i,V_i,W_i \lesssim N^2/K^2\;\;(i=0,1,2)  \\   \Delta \neq 0 }}
2^{2j} \left\lvert \Delta  \right\rvert^{-1}.
\end{equation}
How often can $\Delta$ be small? Let
\begin{equation}\label{titmouse}
H=\gcd(h_1^2 U_1 - h_0^2 U_0,  h_1^2 V_1 - h_0^2 V_0) ,\quad G=\gcd(h_1^2,h_0^2U_0,h_0^2V_0),\quad
T=|h_1^2 U_1 - h_0^2 U_0|.
\end{equation}
After permuting $U,V,$ and $W$ if necessary, and also swapping $(U_1,V_1,W_1)$, $(U_2,V_2,W_2)$ if necessary, we will have
\begin{equation}\label{lapwing}
\begin{gathered}
	\frac{\gcd(h_1^2 U_1 - h_0^2 U_0,  h_1^2 W_1 - h_0^2 W_0)}{ \gcd( h_1^2,h_0^2 U_0,h_0^2 W_0)}\geq\frac{\gcd(h_1^2 U_1 - h_0^2 U_0,  h_1^2 V_1 - h_0^2 V_0)}{ \gcd( h_1^2,h_0^2 U_0,h_0^2 V_0)}=\frac{H}{G},
	\\
	\quad
	|h_2^2 U_2 - h_0^2 U_0|\leq T,
\quad
	|h_i^2 V_i - h_0^2 V_0|
\leq T,\quad
	|h_i^2 W_i - h_0^2 W_0|\leq T\quad (i=1,2).
\end{gathered}
\end{equation}
In particular, if $\Delta\neq 0$ then it follows from this that $T>0$ and $H>0$. Observe that
\begin{multline*}
	\Delta = 
	h_2^2
	\begin{pmatrix}
	(h_1^2W_1-h_0^2W_0)V_2+(h_0^2V_0-h_1^2V_1) W_2\\
	(h_1^2U_1-h_0^2U_0) W_2+(h_0^2W_0-h_1^2W_1)U_2\\
	(h_1^2V_1-h_0^2V_0)U_2+(h_0^2U_0-h_1^2U_1) V_2 	
	\end{pmatrix}
	+ Q(h_0,h_1,U_0,U_1,V_0,V_1,W_0,W_1)
\end{multline*}
for some triple of forms $Q$.
Suppose for a moment we are given $h_i\sim K$ ($i=0,1,2$) and $U_i,V_i,W_i$ ($i=0,1$) satisfying \eqref{titmouse} and \eqref{lapwing}. One can deduce from the formula above that
\begin{equation*}
\sum_{\substack{ U_2,V_2,W_2 \lesssim N^2/K^2 \\ \left\lvert h_2^2(U_2, V_2 ,W_2)-h_0^2 (U_0,V_0,W_0)\right\rvert \leq T \\  \lvert \Delta \rvert \leq M }} 1
\lesssim
1+\frac{N^2H}{K^2T}+\frac{M}{K^2}+\frac{N^2HM}{K^4T}+\frac{M^2N^2}{T^2K^6}.
\end{equation*}
In fact, if $M/TK^2\gtrsim 1$, then there are at most $N^2/K^2$ choices for $U_2$; when $U_2$ is fixed, from $|\Delta|\leq M$ we get that there is at most $M/TK^2$ choices for each of $V_2$ and $W_2$. This gives a total of \[\frac{N^2}{K^2}\bigg(\frac{M}{TK^2}\bigg)^2=\frac{M^2N^2}{T^2K^6}\] choices. 

Now if $M/TK^2\ll 1$, then for fixed $U_2$ there is at most $1$ choice for each of $V_2$ and $W_2$. Moreover, as $|\Delta|\leq M$, there are at most $M/K^2+1$ choices for the quantity
\[(h_1^2V_1-h_0^2V_0)U_2+(h_0^2U_0-h_1^2U_1) V_2,
\] so there are at most $M/K^2+1$ choices for the residue
\[U_2\pmod \xi,\quad \xi:=\frac{h_0^2U_0-h_1^2U_1}{\gcd(h_1^2V_1-h_0^2V_0,h_0^2U_0-h_1^2U_1)}.
\] Now $U_2\lesssim N^2/K^2$ and $|\xi|\sim T/H$, so the number of choices for $U_2$ is at most
\[\bigg(\frac{M}{K^2}+1\bigg)\bigg(\frac{N^2H}{K^2T}+1\bigg)\lesssim 1+\frac{N^2H}{K^2T}+\frac{M}{K^2}+\frac{N^2HM}{K^4T}.
\]

Clearly the same bound holds if we permute $U,V$ and $W$ in \eqref{titmouse}, and also if we put $\sim$ in place of the equalities in \eqref{titmouse}. Inserting this into \eqref{lyrebird} we find
\begin{align*}
	\Sigma_0(\alpha,\beta)
	\preceq_{\alpha,\beta}{}&
	1+
	\sum_{\substack{T,M,G,H\in 2^\Z \\ 1\leq T\leq N^2\\ 1\leq M \leq T^2\\ 1\leq G\leq K^2 \\ G\leq H \leq T}}
	\sum_{\substack{h_0,h_2\sim K \\ U_0\lesssim N^2/K^2 }}
		\sum_{\substack{h_1\sim K  \\  \gcd(h_1^2,h_0^2 U_0) \gtrsim G}}
	\sum_{\substack{ U_1\lesssim N^2/K^2  \\ 
	|h_1^2 U_1 - h_0^2 U_0| \sim T }}
\sum_{V_0,W_0\lesssim N^2/K^2 }
		 \\&
	\sum_{\substack{V_1,W_1 \lesssim N^2/K^2  \\ 
		|h_1^2 V_1 - h_0^2 V_0| \lesssim T \\ |h_1^2 W_1 - h_0^2 W_0| \lesssim T\\
	  \gcd(h_1^2 U_1 - h_0^2 U_0,  h_1^2 V_1 - h_0^2 V_0) \sim H \\
	  \gcd( h_1^2,h_0^2 U_0,h_0^2 V_0) \sim G\\
 	\frac{\gcd(h_1^2 U_1 - h_0^2 U_0,  h_1^2 W_1 - h_0^2 W_0)}{ \gcd( h_1^2,h_0^2 U_0,h_0^2 W_0)}\gtrsim H/G
}}
	\sum_{\substack{ U_2,V_2,W_2 \lesssim N^2/K^2 \\ \left\lvert h_2^2(U_2, V_2 ,W_2)-h_0^2 (U_0,V_0,W_0)\right\rvert \leq T \\  \lvert \Delta \rvert \leq M }} 2^{2j}M^{-1}
	\\
	{}
	\preceq{}&
	1+
	\sum_{\substack{T,M,G,H\in 2^\Z \\ 1\leq T\leq N^2\\ 1\leq M \leq T^2\\ 1\leq G\leq K^2 \\ G\leq H \leq T}}
	N^2
	\frac{K}{G^{1/2}} 
	\left(\frac{T}{K^2}+1\right) \frac{N^4}{K^4}
	\\
	&\left(\frac{GT}{HK^2}+1\right)^2
	\left(1+\frac{HN^2}{TK^2} +\frac{M}{K^2}+\frac{N^2HM}{K^4T}+\frac{M^2N^2}{T^2K^6}\right)
	2^{2j}M^{-1}.
\end{align*}
 Here we will explain the bound in the summation in $(V_1,W_1)$; the bounds in all other summations are straightforward. First consider $V_1$; as $(h_0,h_1,U_0,U_1,V_0,W_0)$ is fixed, the four gcd's in the summation have $\preceq 1$ choices, so we may assume they are fixed. Let
\[\xi= \gcd(h_1^2 U_1 - h_0^2 U_0,  h_1^2 V_1 - h_0^2 V_0)\sim H, \qquad\gcd(\xi,h_1^2)=\gcd( h_1^2,h_0^2 U_0,h_0^2 V_0)\sim G,
\]then we have $h_1^2V_1-h_0^2V_0\equiv 0\pmod\xi$. This implies that the residue of $V_1$ modulo $\xi/\gcd(\xi,h_1^2)\sim H/G$ is fixed; as also $|h_1^2 V_1 - h_0^2 V_0| \lesssim T$, the number of choices for $V_1$ will be at most $GT/HK^2+1$. For $W_1$ the bound is similar, except that instead of $\xi/\gcd(\xi,h_1^2)$ we have 
\[\frac{\eta}{\gcd(\eta,h_1^2)}=\frac{\gcd(h_1^2 U_1 - h_0^2 U_0,  h_1^2 W_1 - h_0^2 W_0)}{ \gcd( h_1^2,h_0^2 U_0,h_0^2 W_0)},\qquad \eta:=\gcd(h_1^2 U_1 - h_0^2 U_0,  h_1^2 W_1 - h_0^2 W_0),
\] which is not less than $H/G$ by our assumptions.

Maximising the summand on the right-hand side over $G$ and $M$, we find it is bounded by
$$
N^2
K
\left(\frac{T}{K^2}+1\right) \frac{N^4}{K^4}
\left(\frac{T^2}{H^{1/2}K(H+K^2)^{3/2}}+1\right)
\left(
1+\frac{HN^2}{TK^2}+\frac{N^6}{T^2K^6}
\right)
2^{2j}
$$
and maximising this over $H$, it in turn is bounded by
$$
N^2
K
\left(\frac{T}{K^2}+1\right) \frac{N^4}{K^4}
\left(
\frac{N^2 T^{3/2}}{(T^{3/2}+K^3)K^3}
+\frac{T N^2}{K^5}+\frac{T^2}{K^4}+\frac{N^6}{K^{10}}
+\frac{N^2}{K^2}
+\frac{N^6}{T^2K^6}
\right)
2^{2j}
$$
which, considering the cases when $T\in\{1, K^2, N^2\}$, is bounded by
$$
\bigg(\frac{N^{12}}{K^9}+\frac{N^{14}}{K^{15}}\bigg)2^{2j}\lesssim \frac{N^{12}}{K^9}2^{2j}\lesssim\frac{N^{12}}{K^3}2^{3j}
$$
using the fact that $2^j\geq N^{-3}$ and $N^{1/2}\leq K\leq N$.
\end{proof}
\subsection{Completing the case $p=8$}

Combining  \eqref{sandpiper} with \eqref{flamingo} and \eqref{buzzard} shows that
$$
\lVert e^{it\tilde{\Delta}}f \rVert^8_{L^8([0,T]\times \mathbb{T}^2)}
\lesssim
\lvert S^4\rvert \Big((NT+N) +(N^3 +NT)+(N^4+NT)\Big)
$$
which proves the case $p=8$ of Theorem~\ref{handsaw}.

\section{The counting argument through Pall's bound: $p=10$}

\label{section6}

With the same ansatz for $f$ as in Section~\ref{section4}, we will prove

\begin{thm}
Generically in $(\alpha_{ij})$,
$$
\| e^{it \widetilde{\Delta}} f \|_{L^{10}([0,T]\times \mathbb{T}^2)}^{10}
\lesssim
|S^5|(N^6+TN^2).
$$
\end{thm}
This implies Theorem~\ref{thmp8} for $p=10$.

\bigskip

Following the same argument as in Section~\ref{section5}, we have by \eqref{jacana} that
$$
\| e^{it \widetilde{\Delta}} f \|_{L^p([0,T])}^{10}\lesssim
\sum_{2^j > \frac{1}{T}} 2^{-j} \lvert S^5\rvert \sup_{a,b\lesssim N, \tau\lesssim N^2 }\sum_{\substack{A_1,B_1,C_1 \lesssim N^2 \\ |A_1 + B_1 \alpha + C_1 \beta-\tau| < 2^{j}}}  \lvert\Omega^{5,N,+}_{a,b,A_1,B_1,C_1}\rvert
$$
where $\Omega^{5,N,+}_{a,b,A_1,B_1,C_1}$ is the set of solutions of size less than $N$ to 
$$
\sum_{i=1}^5 k_i = a, \;\; \sum_{i=1}^5 \ell_i = b, \;\; \sum_{i=1}^5 k_i^2 = A_1, \;\;\sum_{i=1}^5 \ell_i^2 = B_1,\;\;  \sum_{i=1}^5 k_i \ell_i = C_1.
$$
After shifting the variables we can assume $a,b=0$. Over $\Q$ the form $x_1^2+x_2^2+x_3^2+x_4^2+(x_1+x_2+x_3+x_4)^2$ is equivalent to $x_1^2+x_2^2+x_3^2+5x_4^2$, {namely
\begin{multline*}x_1^2+x_2^2+x_3^2+x_4^2+(x_1+x_2+x_3+x_4)^2\\=(x_1+x_2+x_4/2)^2+(x_2+x_3+x_4/2)^2+(x_1+x_3+x_4/2)^2+5(x_4/2)^2.\end{multline*}} So under some linear change of variables, each integral solution to the equations above gives us an integral solution to
$$
\sum_{i=1}^3 k_i^2+5k_4^2 = A', \;\;\sum_{i=1}^3 \ell_i^2+5\ell_4^2 = B',\;\;  \sum_{i=1}^3 k_i \ell_i +5k_4\ell_4= C'
$$
for some fixed, integral $A',B',C'$. That is,
\begin{align*}
\| e^{it \widetilde{\Delta}} f \|_{L^{10}([0,T])}^{10}
&\lesssim
\sum_{2^j > \frac{1}{T}} 2^{-j} \lvert S^5\rvert \sup_{\tau'\lesssim N^2 }
\sum_{\substack{A',B',C' \lesssim N^2 \\ |A' + B' \alpha + C' \beta-\tau'| < 2^{j}  }}
\sum_{\substack{ \sum_{i=1}^3 \mu_i^2 +5\mu_4^2 = A',\\\sum_{i=1}^3 \lambda_i^2+5\lambda_4^2 = B',\\  \sum_{i=1}^3 \mu_i \lambda_i +5\mu_4\lambda_4= C' }} 1
\\
& \preceq 
\sum_{2^j > \frac{1}{T}} 2^{-j} \lvert S^5\rvert \sup_{\tau'\lesssim N^2 }
\sum_{\substack{A',B',C' \lesssim N^2 \\ |A' + B' \alpha + C' \beta-\tau'| < 2^{j}  }}
\sum_{x,y\lesssim N}
\sum_{\substack{ \sum_{i=1}^3 \mu_i^2 = A'-5x^2,\\\sum_{i=1}^3 \lambda_i^2 = B'-5y^2,\\  \sum_{i=1}^3 \mu_i \lambda_i = C'-5xy }} 1.
\end{align*}
The idea is to use Pall's result to bound the innermost sum. First we deal with degenerate cases. The first degenerate case is when $A'=5x^2$ or $B'=5y^2$; without loss of generality we treat $B'=5y^2$ which gives a contribution of
$$
\sum_{(m,x,y):\lvert m + 5x^2+ 5y^2 \alpha+ 5xy\beta - \tau' \rvert<2^j}
\sum_{(\mu_i):\sum_{i=1}^3 \mu_i^2 = m} 1
\preceq
(2^jN^4+1) N,
$$
by the Diophantine condition $|i+j\alpha+k\beta | \gtrsim N^{-4}$ if $i,j,k \lesssim N^2$, with $(i,j,k) \neq (0,0,0)$. This is more than satisfactory. 

Next we consider the contribution from $(\sum_{i=1}^3 \mu_i^2)(\sum_{i=1}^3 \lambda_i^2)=(\sum_{i=1}^3 \mu_i\lambda_i)^2\neq 0$. In this case there are unique $z_i,p,q$ with $p,q\in\N$ and $\gcd(p,q)=1$ such that $\mu_i=pz_i$, $\lambda_i=qz_i$. These terms give a contribution of
\begin{multline*}
\sum_{\substack{A',B',C' \lesssim N^2 \\ |A' + B' \alpha + C' \beta-\tau'| < 2^{j}  }}
\sum_{\substack{mp^2+5x^2=A',\\ mq^2+5y^2=B',\\ mpq+5xy = C'}}
\sum_{\sum_{i=1}^3 z_i^2 = m} 1
\\
\lesssim
(2^jN^4+1)N\sup_{A',B',C' \lesssim N^2}
\lvert\{ (m,p,q,x,y) : mp^2+5x^2=A',\; mq^2+5y^2=B',\; mpq+5xy = C' \}\rvert.
\end{multline*}
We claim the last supremum is $ \preceq N$. Indeed we can choose $x$ arbitrarily, get $\preceq 1$ choices for $m$ and $p$ by the divisor bound, and then as $mp\neq 0$ by assumption, we get that $(q,y)$ lies on the intersection of an ellipse and a line and we're done. This gives us a contribution of $\lesssim (2^jN^4+1) N^2$, as desired.

Finally we treat the case $0\neq (\sum_{i=1}^3 \mu_i^2)(\sum_{i=1}^3 \lambda_i^2)\neq(\sum_{i=1}^3 \mu_i\lambda_i)^2$. By Pall, the contribution is
\begin{equation*}
\sum_{\substack{A',B',C' \lesssim N^2 \\ |A' + B' \alpha + C' \beta-\tau'| < 2^{j}  }}\sum_{h\lesssim N}\sum_{\substack{ h^2U+5x^2=A\\ h^2V+5y^2=B\\ h^2W+5xy = C }} h.
\end{equation*}
We get
$$
\lvert\{ (U,V,W,x,y) : h^2U+5x^2=A,\; h^2V+5y^2=B,\; h^2W+5xy = C \}\rvert
\preceq \frac{N^2}{h^2}
$$
by {choosing the values of $W$ and $xy$, and then using the divisor bound. Note that if $xy=0$, say $x=0$, then $(U,W)$ is uniquely fixed, and we can get the same bound by choosing the values of $V$ and $y^2$.} This gives us
\begin{equation*}
\sum_{\substack{A',B',C' \lesssim N^2 \\ |A' + B' \alpha + C' \beta-\tau'| < 2^{j}  }}\sum_{h\lesssim N}\sum_{\substack{ h^2U+5x^2=A\\ h^2V+5y^2=B\\ h^2W+5xy = C }} h
\preceq
(2^jN^4+1) N^2
\end{equation*}
and we're done.

\appendix
{\section{Higher dimensions: the result of Guo and Zhang}}

{
The case $k=2$ of the main result of Guo and Zhang~\cite{GZ} is
\begin{multline*}
\idotsint\limits_{\substack{0\leq\beta_{i}\leq 1\\0\leq\alpha_{ij}\leq 1}} \bigg\lvert
\sum_{\substack{ x \in \mathbb{Z}^d \\ |x|\leq M }}
e\Big( \sum_{i,j=1}^d \alpha_{ij} x_i x_j + \sum_{i=1}^d \beta_i x_i\Big)\bigg\rvert^{2s} \,d\alpha_{11} \, d\alpha_{12}\dotsm d\alpha_{dd} \,d\beta_1\dotsm d\beta_d
\\
\lesssim_\epsilon
N^\epsilon
\bigg(
N^{ds}
+\sum_{j=1}^d
N^{(2s-1)j+d-j(j+2)}
\bigg),
\end{multline*}
and the right-hand side is seen to be $\lesssim N^\epsilon ( N^{ds}+ N^{2ds-d(d+2)})$. 
We sketch the argument by which \eqref{jubjub} follows from this  bound. More precisely, if we define
\begin{equation*}
	Z(M)=
\#\Big\{
x^{(1)},\dotsc,x^{(2s)}
\in \mathbb{Z}^d
:
|x^{(i)}|\leq M,
\,
\sum_{i=1}^{2s} (-1)^i x^{(i)}_j x^{(i)}_k=0
\;\; (1\leq j,k\leq d)
\Big\},
\end{equation*}
then we have
\begin{equation}
	\label{bat}
	Z(M)
\lesssim_\epsilon
M^{ds+d+\epsilon}+ M^{2ds-d(d+1)+\epsilon}.
\end{equation}
Indeed this follows from the result of Guo-Zhang on summing over the possible values of $\sum_{i=1}^{2s} (-1)^i x^{(i)}$. We will deduce \eqref{jubjub} from this last estimate.
 }

{
For generic $(\alpha_{ij})$ we have
\begin{equation}\label{greatauk}
I_s(T) \lesssim_{\alpha} N^\epsilon T^{\epsilon} \idotsint\limits_{1\leq\alpha_{ij}\leq 2} I_s(T)\,d\alpha_{11} \, d\alpha_{12}\dotsm d\alpha_{dd},
\end{equation}
by estimating the measure of the set of exceptional $\alpha$ for which this fails, and taking the union over  $T,N$ of the form $2^k$.
 For $T\geq 2$ the right-hand side above is
\begin{equation}
\asymp
N^\epsilon
T^{1+\epsilon}
\idotsint\limits_{0\leq\alpha_{ij}\leq 1}
\sup_{x\in \T^d}\left\lvert
K^{(d)}_N(1,x)\right\rvert^{2s} \, d\alpha_{11} \, d\alpha_{12}\dotsm d\alpha_{dd},
\label{hamerkop}
\end{equation}
using the fact that the last integrand is 1-periodic in each $\alpha_{ij}$. Now, by the same argument as \eqref{bowerbird}, we have
\[
|K_N^{(d)}(1,x)|^2 \preceq N^d \left[ \# \{ n \in [-4N,4N]^d \;\; \mbox{s.t.} \;\; \| L_j(n) \| < \frac{2}{N} \} + 1 \right].
\]
The idea is that we can bound the right-hand side by some kind of exponenetial sum which looks very much like $|K_N^{(d)}(0,0)|^2$.
To motivate this strategy requires some inspection of the proof of \eqref{bowerbird}, but the argument itself is short. We observe that the right-hand side in the last display is
\[
\lesssim
N^d \sum_{{w},{n}\in \Z^d} e^{-\frac{2}{N^2}|n|^2-\frac{N^2}{8}\left|w-4( \alpha_{ij})n \right|^2}
\]
and, by Poisson summation in $w$, this is 
\[
=
2^d
\sum_{n,m \in \mathbb{Z}^d }
e^{-\frac{2}{N^2}|n|^2-\frac{2}{N^2}|m|^2}
e\bigg(-4  \sum_{ij} \alpha_{ij} n_im_j \bigg).
\]
Changing variables to $x=n-m, y=n+m$ we have shown that
\[
\sup_{x\in \T^d}\left\lvert
K^{(d)}_N(1,x)\right\rvert^2
\preceq
\sum_{\substack{ x,y \in \mathbb{Z}^d \\ x\equiv y \,(2)} }
e^{-\frac{|x|^2+|y|^2}{N^2}}
e\bigg( \sum_{ij} \alpha_{ij} x_i x_j - \sum_{ij} \alpha_{ij} y_i y_j  \bigg).
\]
Substituting this into \eqref{hamerkop} we obtain
\begin{align*}
I_s(T)
&\preceq_{\alpha}
T
\sum_{\substack{ x^{(i)},y^{(i)} \in \mathbb{Z}^d ,\, x^{(i)}\equiv y^{(i)} \,(2) \\ \sum_{i=1}^s x^{(i)}_jx^{(i)}_k = \sum_{i=1}^s y^{(i)}_jy^{(i)}_k  \, (1\leq j,k\leq d) } }
e^{-\sum_i\frac{|x^{(i)}|^2+|y^{(i)}|^2}{N^2}}
\\&\lesssim
T\sum_{\substack{M\in 2^\Z\\ M \geq N}}
e^{-M^2/N^2} Z(M).
\end{align*}
This together with \eqref{bat} proves \eqref{jubjub}. 
}

We would have an optimal bound for $Z(M)$ if we could replace the $ds+d$ in \eqref{bat}  and hence \eqref{jubjub} with $ds$. This would be predicted by a square-root cancellation heuristic standard in the circle method, and it seems likely that it can be proved as follows. We can intepret $Z(M)$ as a count of $d$-tuples of integer points $y\in \Z^{2s}$ with norm at most $M$ and whose span $V$ is contained in the hypersurface $\sum (-1)^i y^{(i)}=0$. Given $\dim V$, we can count the possible spans $V$, for example using the work of Franke-Manin-Tschinkel~\cite{FMT}, and then we can count $d$-tuples of points in each space $V$.

\end{document}